\documentclass[reqno]{amsart}
\usepackage{graphicx}\usepackage{amsfonts} \usepackage{dsfont}
\usepackage{amsmath,amssymb}
\usepackage[rgb]{xcolor}
\usepackage{hyperref}
\usepackage[inline]{enumitem}
\newtheorem{thm}{Theorem}[section]

\newtheorem{lem}[thm]{Lemma}
\newtheorem{prop}[thm]{Proposition}
\newtheorem{que}{Question}
\theoremstyle{definition}

\theoremstyle{remark}
\newtheorem{rem}[thm]{Remark}
\numberwithin{equation}{section}

\newcommand{\abs}[1]{\left\vert#1\right\vert}

\newcommand{\realR}{\mathbb{R}}

\newcommand{\bigO}{\mathcal{O}}
\renewcommand{\Re}{\operatorname{Re}}
\renewcommand{\Im}{\operatorname{Im}}

\begin{document}
\title[Phase Transitions]{Phase transitions  for infinite products of large non-Hermitian random matrices}

\author{Dang-Zheng Liu}
\address{Key Laboratory of Wu Wen-Tsun Mathematics, CAS, School of Mathematical Sciences, University of Science and Technology of China, Hefei 230026, P.R.~China}
\email{dzliu@ustc.edu.cn}

\author{Yanhui Wang}
\address{School of Mathematics and Statistics, Henan University, Henan, 475001, P.R.~China}
\email{yhwang@henu.edu.cn}


\begin{abstract}
Products of $M$ i.i.d. non-Hermitian random  matrices of size $N \times N$ relate  Gaussian fluctuation of  Lyapunov and stability exponents in dynamical systems (finite $N$ and large $M$) to local eigenvalue  universality in random matrix theory (finite $M$ and large $N$).
 The remaining task is to study  local eigenvalue  statistics as  $M$ and $N$  tend to infinity simultaneously,  which  lies   at the heart of understanding  two kinds of universal patterns.  For   products of  i.i.d. complex Ginibre matrices, truncated unitary matrices and spherical ensembles,  as $M+N\to \infty$ we prove  that  local  statistics   undergoes a transition when  the relative  ratio $M/N$ changes from $0$ to $\infty$:  Ginibre statistics when $M/N \to  0$, normality  when $M/N\to \infty$, and new critical phenomena when  $M/N\to \gamma \in  (0, \infty)$.  

\end{abstract}
\maketitle
\tableofcontents
\section{Introduction and main results}

\subsection{Lyapunov and stability exponents}


The study on  products of random matrices dates back  at least to  the seminal articles by Bellman \cite{Bellman54} in 1954 and by Furstenberg and Kesten \cite{Furstenberg-Kesten60} in 1960.  The asymptotic results about  products  of random matrices discovered  by Furstenberg and Kesten \cite{Furstenberg-Kesten60}, as generalizations  of  the law of large numbers and central limit  theorem  in probability theory,   
have initiated great interest in the area, see \cite{Cohen-Kesten-Newman86} and references therein for some early articles.   Recently,  a lot of  important applications are found  in Schr\"{o}dinger operator theory \cite{Bougerol-Lacroix85}, in statistical physics relating to disordered and chaotic dynamical systems \cite{Crisanti-Paladin-Vulpiani93}, in wireless communication \cite{Tulino-Verdu04} and in free probability theory \cite{Mingo-Speicher17}.

%
  Specifically, let  $X_{1}, X_2, \dotsc, X_{M}$ be i.i.d. random matrices of size ~$N\times N$ and  their product
  \begin{equation} \label{Mproduct}
  \Pi_{M}^{}=X_{M}\cdots X_{2}X_{1}, \end{equation}
   if $\mathbb{E}( \log^{+}\|X_{1}\|)<\infty$,
  then  the celebrated result \cite[Theorem 2]{Furstenberg-Kesten60}    shows that for any fixed $N$ the largest Lyapunov exponent, defined as
\begin{equation} \label{LE2}
  \lambda_{1}:=\lim_{M\rightarrow \infty} \frac{1}{M} \log\|\Pi_{M}\|,
\end{equation}
exists with probability $1$.  Furthermore,  by the Oseledets multiplicative ergodic theorem  (see \cite{Oseledec68} or  \cite{Raghunathan79}) all Lyapunov exponents (also called the Lyapunov spectrum) $ \lambda_{k}:=\lim_{M\rightarrow \infty}\lambda_{k,M}$ exist   with probability $1$ where the finite-time  Lyapunov exponents
\begin{equation} \label{LE3}
  \lambda_{k,M}:= \frac{1}{2M}\log\!\left( k\mathrm{^{th}\,   largest\,  eigenvalue\,   of\, } \Pi_{M}^{*}\Pi_{M}\right),\quad k=1,2, \ldots,N.
\end{equation}
Lyapunov exponents    play a key role   in   dynamical systems, products of random matrices and random maps, and spectral theory of random Schr\"{o}dinger operators;  see e.g.~\cite{
 Viana14,Wilkinson17}.
Particularly,   when all $X_j$ are independent real/complex Ginibre matrices, that is, with i.i.d.~standard real/complex Gaussian entries,     Newman \cite{Newman86} (real case, $\beta=1$) and Forrester \cite{Forrester13, Forrester15} (real and complex cases with $\beta=1, 2$ respectively)  have calculated  the Lyapunov spectrum as
\begin{equation}
  \lambda_{k} = \frac{1}{2} \Big( \log \frac{2}{\beta} +  \psi\big(\frac{\beta}{2}(N-k+1)\big)\Big), \quad k=1, \dotsc, N,
\end{equation}
where $\psi(z)=\Gamma'(z)/\Gamma(z)$ denotes the digamma function.  See  also \cite{Akemann-Burda-Kieburg14,Ipsen15,Kargin14,Reddy16}.

As  mentioned previously,   the Lyapunov exponents are defined in terms of singular values of  the product $\Pi_M$. But for complex  eigenvalues $z_1, \ldots,z_N$ of  $\Pi_M$ with $|z_1|\geq |z_2|\geq \cdots \geq  |z_N|$,  the finite-time stability exponents are defined as $  \log |z_k|^{1/M}$. Correspondingly,
 the  stability exponents refer to \begin{equation} \label{SE}
  \mu_{k}:=\lim_{M\rightarrow \infty} \frac{1}{M}  \log |z_k|,\quad k=1,2, \ldots,N,
\end{equation}
if they exist   with probability $1$.   This notion
was  first introduced  in the setting of dynamical systems in \cite{Gol} and therein
the question of the asymptotic equality between the finite-time Lyapunov and stability
exponents was  investigated (see also \cite[p.21]{Crisanti-Paladin-Vulpiani93}).
 In the recent articles \cite{Akemann-Burda-Kieburg14,Ipsen15, Forrester15}, for some special  cases that  $X_k$'s in \eqref{Mproduct}  are  real/complex Ginibre matrices, or truncated unitary matrices, the  question can be answered based on exact  asymptotic analysis of the joint eigenvalue  distribution of  $\Pi_M$.   Actually, Guivarc'h in an earlier article  \cite[Theorem 8]{Gui}  has verified   the equality of  Lyapunov  and stability exponents in a more general setting; see also  \cite {Reddy16}   for  a different derivation   in products of  isotropic random matrices.

 As to fluctuations of finite-time  Lyapunov  and stability exponents,   for any fixed $N$ and as $M\rightarrow \infty$, Akemann, Burda and Kieburg proved in \cite{Akemann-Burda-Kieburg14} that $N$  finite-time stability exponents  of  $\Pi_M$ are asymptotically independent Gaussian random variables.  Recently, Reddy has proved the Gaussian fluctuations for products of  isotropic random matrices; see  \cite[Theorem 4.1]{Reddy16}.

\subsection{Universality of non-Hermitian random matrices}
Unlike  the Lyapunov and stability exponents, which are studied in    the situation that  $M \to \infty$  and $N$ is kept fixed,
recently great  interest  in products of   random matrices has been    
   in the  opposite situation that $N \to \infty$  and $M$ is kept fixed.  See \cite{Akemann-Ipsen15, burda2013free} and references therein.

  Historically, the study of  a single non-Hermitian random matrix ($M=1$)  was first initiated  by  Ginibre  \cite{ginibre1965statistical}  for random matrices with i.i.d.   real, complex and quaternion Gaussian entries, and then was extended to i.i.d.  case.   On a macroscopic  level,     the limiting   empirical spectral distribution of  a non-Hermitian  random  matrix  with  i.i.d.   entries  under the certain  moment assumptions is governed by   the famous \textit{circular law}; see   e.g. Bai \cite{bai1997circular},  Girko \cite{girko1985circular},   G\"{o}tze and Tikhomirov \cite{gotze2010circular}, Pan and Zhou  \cite{pan2010circular},  and Tao and Vu \cite{tao2008random,tao2010random2}.
  However, on a   microscopic level,   some  finer structures of   local  eigenvalue statistics  are revealed    first  for  real and complex Ginibre ensembles (see  \cite{BS09,forrester1999exact}), or random truncated orthogonal and  unitary matrices  (see \cite{zyczkowski2000truncations, khoruzhenko2011,KSZ10})
   in the bulk and at the  soft edge,  with the help of the exact eigenvalue density.   These microscopic behaviors are conjectured to be true even  in  the i.i.d.  case,  although the proof  seems much more difficult  than the corresponding  Hermitian analogy.   In \cite{tao2015random},  Tao and Vu  established  a four moment match theorem as a non-Hermitian version of   \cite{Tao-Vu11a}. 
   In a recent article  \cite{CES},
Cipolloni, Erd\H{o}s and  Schr\H{o}der investigate   local universality at the edge hold for general i.i.d. case.

  For any finite and fixed $M$,  the limiting    empirical spectral distribution
  for    products of independent random  matrices 
    has been extensively studied  in  \cite{adhikari2013determinantal,AB12,ABKN14, burda2010spectrum,burda2010eigenvalues, GNT, gotze2010asymptotic, gotze2015Asymptotic,  LiuWang16, Nem17,o2011products} and references therein (the collection of references above  is  far from complete).
     However,  the  local eigenvalue statistics is only known for  products of complex Ginibre matrices  and of truncated unitary matrices (see \cite{AB12} \cite{ABKN14} \cite{LiuWang16}), or  for   products of random  matrices  with i.i.d. entries under a moment matching hypothesis (see \cite{KOV18}).

\subsection{Problem statement--double limit}
As noted before,  under the two different limits of $M\to \infty$ and $N\to \infty$ the local eigenvalue  statistics for products of random matrices  display Gaussian and Ginibre statistics,  respectively.
 Then  a very natural question arises:
 \begin{center}{ \bf What  happens when both $M$ and $N$  tend to infinity?} \end{center}
 Obviously,  this  question  lies at the heart of understanding both kinds of universal limits.

In fact, a similar  double limit problem on singular value statistics for products of random matrices has been   proposed by Akemann, Burda and Kieburg \cite[Section 5]{Akemann-Burda-Kieburg14}    and  Deift \cite{Deift17}.
  In the case of complex Ginibre  matrices,  the  two authors with   D. Wang  solved it completely and proved in \cite{LWW}  that the local singular value statistics undergoes a transition as the relative ratio  $M/N$ changes from 0 to $\infty$: GUE statistics when $M/N \to 0$, Gaussian fluctuation when $M/N\to \infty$, and new critical phenomena when $M/N \to \gamma\in  (0, \infty)$.  This phase transition  is also independently observed by  Akemann, Burda, and Kieburg in the physical language  \cite{Akemann-Burda-Kieburg18}.
See a  few recent articles  \cite{Ahn,GS,HN} for singular values of products of random matrices as $M$ and $N$    change simultaneously.
    As to   the products of non-Hermitian
    complex Ginibre (or  truncated unitary) matrices, Jiang and Qi  studied    absolute values of  $N$ complex eigenvalues, and proved a phase transition of  the largest  absolute value  as  $M/N$ changes from 0 to $\infty$: Gumbel distribution  when $M/N \to 0$, Gaussian  distribution when $M/N\to \infty$, and an interpolating distribution  when $M/N \to \gamma\in  (0, \infty)$; see \cite{JQ17,QX}.  Also, Jiang and Qi  proved the convergence of  the empirical distributions of  complex eigenvalues for the product ensembles,  after proper rescalings; see \cite{JQ19} or \cite{CLQ}.

   The  main goal of this  article  is to  study local statistical properties of complex eigenvalues for products of i.i.d.  non-Hermitian random matrices,  at least including  complex Ginibre  matrices, truncated unitary matrices and spherical ensembles, as $M$ and $N$ may tend to infinity  simultaneously.   We observe  a phase transition of complex eigenvalues when    $M/N$ goes  from 0 to   $\infty$,  which can be treated as an  analogue  of  singular values for products of Ginibre matrices   \cite{Akemann-Burda-Kieburg18,LWW}.  Especially  in the critical regime where  $M/N \to \gamma\in  (0, \infty)$,  we find two  interpolating correlation kernels  between the Gaussian phenomenon and Ginibre statistics  in the bulk and at the soft edge respectively.  Besides, we believe that  the approach we take here   will be useful in      solving similar problems.

%
%

\subsection{Main results} \label{sect1.4} When  $X_{1}, X_2, \dotsc, X_{M}$ in the product $\Pi_M$   defined by   \eqref{Mproduct} are i.i.d. complex Ginibre matrices,
  the  eigenvalues of $\Pi_M$  has been proved by Akemann and Burda  \cite{AB12} to   form a determinantal point process in the complex plane  with  correlation kernel
  \begin{align}
    K_{M,N}(z_{1}, z_{2}) &= \frac{1}{\pi} \sqrt{w(z_{1}) w( z_{2})}\,  T_{M,N}(z_1,z_2), \label{ginibrekernel}
  \end{align}
  where the weight function
  \begin{align}
    w(z) &= \int_{-c - i \infty}^{-c + i \infty} \frac{ds}{2 \pi i} (\Gamma(-s))^{M} \abs{z}^{2s}, \quad z\in \mathbb{C} \label{ginibreweight}
  \end{align}
  with some $c>0$ and the finite sum of a truncated series
  \begin{align}
    T_{M,N}(z_1,z_2) &= \sum_{j=0}^{N-1} \frac{(z_{1}\overline{z}_{2})^{j}}{(\Gamma(j+1))^{M}}. \label{ginibresum}
  \end{align}
 Recall that for a determinantal point process with correlation kernel $K_{M,N}(z_1, z_2)$, the $n$-point correlation functions are  given  by
\begin{equation}
  R^{(n)}_{M,N}(z_1, \ldots, z_{n})= \det[  K_{M,N}(z_i, z_j) ]_{i,j=1}^n,
\end{equation}
see e.g.~\cite{Forrester10,Mehta04}.

For a single matrix  where $M=1$,  limiting  correlation functions exist   in the bulk
      of  $ \abs{u} < 1$
         \begin{multline} \label{Ginibrebulk}
          \lim_{N \to \infty} R^{(n)}_{1,N}\Big(\sqrt{N}(u + \frac{v_{1}}{ \sqrt{N}}),\ldots, \sqrt{N}(u + \frac{v_{n}}{ \sqrt{N}})\Big) =  \det\!{\big[\mathrm{K^{(bulk)}_{Gin}}(v_i,v_j)\big]}_{i,j=1}^n         \end{multline}
     with
   \begin{equation}\mathrm{K^{(bulk)}_{Gin}}(v_1,v_2)=
  \frac{1}{\pi} e^{-\frac{1}{2} (\abs{v_{1}}^{2} + \abs{v_{2}}^{2} - 2v_{1}\overline{v}_{2})},
   \end{equation}
   and  at the edge of  $u= e^{i\theta}$ with $\theta\in(-\pi,\pi]$,
    \begin{multline} \label{Ginibreedge}
          \lim_{N \to \infty} R^{(n)}_{1,N}\Big(\sqrt{N}(u + \frac{v_{1}}{ \sqrt{N}}),\ldots, \sqrt{N}(u + \frac{v_{n}}{ \sqrt{N}})\Big) =  \det\!{\big[\mathrm{K^{(edge)}_{Gin}}(v_i,v_j)\big]}_{i,j=1}^n         \end{multline}
     with
   \begin{equation}\mathrm{K^{(edge)}_{Gin}}(v_1,v_2)=
  \frac{1}{2\pi} e^{-\frac{1}{2} (\abs{v_{1}}^{2} + \abs{v_{2}}^{2} - 2v_{1}\overline{v}_{2})} \mathrm{erfc}\big(\frac{1}{\sqrt{2}}(e^{-i\theta}v_{1}+e^{i\theta}\overline{v}_{2} )\big),
   \end{equation}
   where \begin{equation}\mathrm{erfc}(x)=\frac{2}{\sqrt{\pi}}\int_{x}^{\infty}e^{-t^2}dt;\end{equation}
    see e.g. \cite{BS09,forrester1999exact}.
   Moreover, both hold uniformly for  $v_{1}, \cdots, v_{n}$ in any compact subset of $\mathbb{C}$.

For any finite and fixed $M$, the  scaling limits displayed in \eqref{Ginibrebulk} and  \eqref{Ginibreedge}  have been  proved to  be still valid; see   \cite{AB12,LiuWang16}.  Even    when $M$ goes to infinity but grows much  more slowly than $N$, the same results  will  still be seen to  hold.   Actually, we will establish  a complete  characterization  of limiting correlation functions   for  three product models: products of complex Ginibre matrices,   products of truncated unitary matrices,   and products of  Ginibre and inverse Ginibre matrices, according to  the relative rate of  $M$ and $N$ as $M+N\to \infty$. 
 The first of them is defined  as above and is  renamed   for reference purposes  to
\begin{quote}

{\bf Model A}:  $\Pi_M=X_{M}\cdots X_{1}$   where  $X_{1}, \dotsc, X_{M}$ are  i.i.d.  complex  Ginibre matrices of size $N\times N$, each having  i.i.d.~standard complex Gaussian entries.
\end{quote}

For Model A, our main results  are stated  in the following three theorems. The first one is

  \begin{thm}\label{ginibre-1}   Assume that  $\displaystyle\lim_{M+N \to \infty} M/N=\infty$.
    Given $k \in \{1, \ldots, N\}$, let $\rho = 2/\sqrt{N \psi'(N - k + 1)}$ and
 for $\theta\in (-\pi,\pi]$  set   \begin{align}
      z_{j} &= e^{M (\frac{1}{2} \psi(N - k + 1) + \frac{v_j}{ \rho\sqrt{MN}})} e^{i (\theta + \phi_j)}, \quad j=1,\ldots,n.
    \end{align}
   Then   limiting  $n$-correlation functions for eigenvalues of $\Pi_M$ in Model A
    \begin{align}
      \lim_{M+N \to \infty}
        \Big(\frac{\sqrt{M} }{\rho\sqrt{ N}}\Big)^{n}  \abs{z_{1}\cdots z_{n}}^2 R^{(n)}_{M,N}(z_{1}, \ldots,z_{n})= \begin{cases}\frac{1}{(\sqrt{2\pi})^3}
        e^{-  \frac{1}{2} v_{1}^{2} }, &n=1,\\
       0, & n>1, \end{cases}
    \end{align}
     hold   uniformly for $v_1, \ldots,v_n$ in a compact subset of $\mathbb{R}$ and $\phi_1, \ldots,\phi_n\in (-\pi,\pi]$.
  \end{thm}

In order to state the critical  results, we need to introduce  two new  kernels which correspond to the bulk and edge limits respectively.  For $\beta\in (0,\infty)$ and  $z_1,z_2\in \mathbb{C}$,  define \begin{equation}
    \mathrm{K^{(bulk)}_{\mathrm{crit}}}(\beta;z_1,z_2)=\frac{1}{\sqrt{2 \pi^{3}\beta}} e^{- \frac{1}{2\beta} (|z_1|^2+|\bar{z}_2|^2  + z_1^{2} + \bar{z}_2^2)} \sum_{j= - \infty}^{\infty} e^{- \frac{1}{2} \beta j^{2} -  (z_1+ \bar{z}_2)j}, \label{critbulk}    \end{equation}
and
 \begin{equation}
    \mathrm{K^{(edge)}_{\mathrm{crit}}}(\beta;z_1,z_2)=\frac{1}{\sqrt{2 \pi^{3}\beta}} e^{- \frac{1}{2\beta} (|z_1|^2+|\bar{z}_2|^2  + z_1^{2} + \bar{z}_2^2)}\sum_{j=0}^{\infty} e^{- \frac{1}{2} \beta j^{2} -  (z_1+ \bar{z}_2)j}. \label{critiedge}   \end{equation}
 It is worth emphasizing  that the summation in \eqref{critbulk} is equal to $ \theta(\frac{i(z_1+ \bar{z}_2)}{2\pi},\frac{i\beta}{2\pi})$ where  the Jacobi's basic  theta function
  \begin{equation}
    \theta(z,\tau)= \sum_{j= - \infty}^{\infty} e^{i\pi j^{2}\tau +2i\pi j z},  \quad z\in \mathbb{C}\  \&\  \Im \tau>0. \end{equation}
  Using the Jacobi's theta function identity  (see e.g. \cite[Chapter 1.7]{Mum}) \begin{equation} \theta\big(\frac{z}{\tau}, -\frac{1}{\tau}\big)=
  \sqrt{-i \tau}e^{   \frac{i\pi z^2 }{\tau} } \theta(z,\tau),\end{equation}
 it is  easy to see that  the bulk critical kernel satisfies  a  duality relation
    \begin{equation}
 \beta^{\frac{3}{4}} e^{ \frac{1}{4\beta} (z_1- \bar{z}_2)^2}     \mathrm{K^{(bulk)}_{\mathrm{crit}}}(\beta;z_1,z_2)=\beta'^{\frac{3}{4}} e^{ \frac{1}{4\beta'} (z'_{1}- \bar{z}'_{2})^2}      \mathrm{K^{(bulk)}_{\mathrm{crit}}}(\beta';z'_1,z'_2), \label{bulkduality}    \end{equation}
where
\begin{equation}\beta'=\frac{4\pi^2}{\beta}, \quad z_{1}'=-\frac{2\pi i}{\beta}z_1,  \quad  z_{2}'=-\frac{2\pi i}{\beta}z_2.\end{equation}

Let $\lfloor x \rfloor$ denote the largest integer less than or equal to real $x$. The next two theorems are

\begin{thm}\label{ginibre-2} Assume that   $\displaystyle\lim_{M+N \to \infty} M/N=\gamma\in (0,\infty)$.
  Given   $q \in [0, 1)$, let $u_{N} = \abs{u_{N}} e^{i \theta}$ with $\theta \in (- \pi, \pi]$ such that $\abs{u_{N}}^{2} = 1 - \frac{1}{N}\lfloor qN \rfloor$,
   set  \begin{align}
      z_{j} = \big(\sqrt{N}u_{N}\big)^{M} \Big(1  + \frac{1}{\sqrt{\gamma MN}} \big(v_j-\frac{\gamma}{4(1-q)}\big)\Big)^{M}, \quad j=1,\ldots,n.
    \end{align}
 Then   limiting  $n$-correlation functions for eigenvalues of $\Pi_M$ in Model A
     \begin{multline}
      \lim_{M+N \to \infty}
        \abs{z_{1}\cdots z_{n}}^2 R^{(n)}_{M,N}(z_{1}, \ldots,z_{n}) \\
       =\begin{cases}
        \det\!\big [  \mathrm{K^{(bulk)}_{\mathrm{crit}}}(\frac{\gamma}{1-q};v_i, v_j)\big]_{i,j=1}^n, & q \in (0, 1),\\
       \det\!\big [  \mathrm{K^{(edge)}_{\mathrm{crit}}}(\gamma;v_i, v_j)\big]_{i,j=1}^n, & q = 0,
      \end{cases}
    \end{multline}
     hold   uniformly for $v_1, \ldots,v_n$ in a compact subset of $\mathbb{C}$.
   \end{thm}

\begin{thm}\label{ginibre-3} Assume that   $\displaystyle\lim_{M+N \to \infty} M/N=0$.
  Given   $ u= \abs{u} e^{i \theta}$ with $\theta \in (- \pi, \pi]$,
   set  \begin{align}
      z_{j} = \big(\sqrt{N}u\big)^{M} \Big(1  + \frac{v_j}{u\sqrt{ MN}}\Big)^{M}, \quad j=1,\ldots,n.
    \end{align}
 Then   limiting  $n$-correlation functions for eigenvalues of $\Pi_M$ in Model A
     \begin{multline}
      \lim_{M+N \to \infty}
        \Big(\frac{M}{N|u|^2}\Big)^{n} \abs{z_{1}\cdots z_{n}}^2   R^{(n)}_{M,N}(z_{1}, \ldots,z_{n}) \\
       =\begin{cases}
        \det\!\big [  \mathrm{K^{(bulk)}_{\mathrm{Gin}}}(v_i, v_j)\big]_{i,j=1}^n, & |u| \in (0, 1),\\
       \det\!\big [  \mathrm{K^{(edge)}_{\mathrm{Gin}}}(v_i, v_j)\big]_{i,j=1}^n, & |u| = 1,
      \end{cases}
    \end{multline}
    hold   uniformly for $v_1, \ldots,v_n$ in a compact subset of $\mathbb{C}$.
   \end{thm}

The rest of this article is organized as follows. In the next Section \ref{sect2} we  give proofs of the above three theorems.  In  Sections \ref{sect3} and   \ref{sect4},  we investigate   products of truncated unitary matrices and products of Ginibre and inverse Ginibre matrices, which includes products of spherical ensembles  as   a particular case.  In last Section \ref{conclusion},       we discuss a few relevant questions such as crossover transitions from critical bulk and edge kernels and open questions.

\section{Proofs of Theorems  \ref{ginibre-1}-\ref{ginibre-3}}\label{sect2}
Theorems \ref{ginibre-1},   \ref{ginibre-2} and \ref{ginibre-3} reflect distinct local  behavior of eigenvalues as the relative ratio of  $M$ and $N$ changes, so we need to  take different routes to complete their proofs.   Meanwhile, we will understand the reason behind it.   It is worth emphasizing that our method will be useful in the study of local universality problem for  a class of non-Hermitian random matrices where joint  eigenvalue  probability density functions can be given explicitly, especially for  products of non-Hermitian random matrices.

Since asymptotic analysis   of  $n$-correlation functions for eigenvalues of $\Pi_M$ in Model A reduces
to that of  the correlation kernel \eqref{ginibrekernel}, it is sufficient for  us to focus on the latter. The kernel consists of two factors that are  the weight function \eqref{ginibreweight} and a finite sum \eqref{ginibresum},  we first discuss  them separately  and then  combine them together to complete the proof.


  We  need  the following technical lemma  and will  use it  frequently.
  \begin{lem}\label{lem-g}
    For a real   parameter  $a\geq 1$,  introduce a function of $z$
    \begin{align}
      g(a;z) = a\big(\log \Gamma(a - z) - \log \Gamma(a) + z\psi(a)\big), \quad \Re z <a,
    \end{align} where $\psi(a)$ denotes the digamma function.
    Let  $\delta\in [0, 1/2]$ and $y_0\in [-\delta,\delta]$,  then  the following hold.
    \begin{enumerate}
      \item  \label{lem-g1}  $\Re g(a;x+iy_0)$ as a function of $x$ is  strictly decreasing over  $(-\infty, 0]$ and strictly increasing over  $[\delta, a-\delta]$.

      \item \label{lem-g2} For any $y\in \realR$,
          \begin{equation}
            \Re g(a;iy) \leq - \frac{|y|}{4} \min\{a,|y|\}.
          \end{equation}

      \item \label{lem-g3} For sufficiently large $a$,
          \begin{equation}
            g(a; z) = \frac{1}{2} z^{2}+ \bigO(a^{-\frac{1}{4}}), \quad \mathrm{if}\  z=\bigO(a^{\frac{1}{4}}),
          \end{equation}
          and
          \begin{equation}
            \inf_{x\leq a-\delta, \abs{x} \geq a^{\frac{1}{4}}} \Re g(a;x+iy_0) \geq \frac{\sqrt{a}}{4}.
          \end{equation}
    \end{enumerate}
  \end{lem}
  \begin{proof}
    Using the series expansion \cite[5.7.6]{Olver10} for $z\neq 0, -1,-2,\ldots$
    \begin{equation} \label{eq:digamma_series}
      \psi(z) = -\gamma_0 + \sum^{\infty}_{n = 0} \left( \frac{1}{n + 1} - \frac{1}{n + z} \right)
    \end{equation}
    with $\gamma_0$ the Euler constant, we have
    \begin{equation}
      \begin{split}
        \frac{d}{dx}\Re g(a;x+i y_0) &= a\big(\psi(a) - \Re \psi(a - x-i y_0)\big) \\
        &= \sum_{n=0}^{\infty} \frac{a(x(a+n-x)-y_{0}^2)}{(a+n)((a+n-x)^2+y_{0}^2)}
      \end{split}
    \end{equation}
    and further
    \begin{align}
      \frac{d}{dx}\Re g(a;x+i y_0) \begin{cases}<0, &x<0,\\
       >0, & \delta\leq x\leq a-\delta. \end{cases}
         \end{align}
    So part \ref{lem-g1}  immediately follows.

    For part \ref{lem-g2}, we see  from \cite[5.8.3]{Olver10} that
    \begin{equation}
      \begin{split}
        \Re{g(a; iy)}
        = - \frac{a}{2} \sum_{n=0}^{\infty} \log\Big(1 + \frac{y^{2}}{(a+n)^{2}}\Big)
        < - \frac{a}{2} \int_{a}^{\infty} \log\Big(1 + \frac{y^{2}}{t^{2}}\Big) dt.
      \end{split}
    \end{equation}
    For $\abs{y} \leq a$, we use the inequality $\log(1 + z) \geq z/2$ for $z \in [0, 1]$ to obtain
    \begin{align}
      \Re{g(a; iy)} < -\frac{a}{4} \int_{a}^{\infty} \frac{\abs{y}^2}{t^{2}} dt = -\frac{y^{2}}{4},
    \end{align}
    while for  $\abs{y} > a$,
    \begin{align}
      \Re{g(a;iy)}
      < - \frac{a}{2} \int_{|y|}^{\infty} \log\Big(1 + \frac{y^{2}}{t^{2}}\Big) dt
      < -\frac{a\abs{y}}{4}.
    \end{align}
    Put them together and we  thus  get the desired result.

    Finally, for part \ref{lem-g3}, take the Taylor expansion for $g(a;z)$ at $z= 0$ and use approximation for the digamma function we  easily see  for $z=\bigO(a^{\frac{1}{4}})$
    \begin{align}
      g(a;z) = \frac{z^{2}}{2} + \bigO(\frac{z^{3}}{a}) = \frac{z^{2}}{2} + \bigO(a^{-\frac{1}{4}}). \label{gleading}
    \end{align}
    Thus we know from the monotonicity in part \ref{lem-g1} and \eqref{gleading} that
    \begin{equation}
      \inf_{x\leq a-\delta, \abs{x} \geq a^{\frac{1}{4}}} \Re g(a;x+iy_0)=\Re g(a;\pm a^{\frac{1}{4}}+iy_0)  \geq \frac{\sqrt{a}}{4}.
    \end{equation}
  \end{proof}

\subsection{Proof of Theorem \ref{ginibre-1}}
  By the assumption on $M$ and $N$ in this setting, we know that $M$ must tend to infinity and $N$ may tend to infinity or be a finite number, whenever  $M+N\to \infty$.

  \begin{proof}[Proof of Theorem \ref{ginibre-1}]

    We proceed in two steps, in which the weight function and finite sum are discussed  respectively, and then  combine them to complete the proof.

    {\bf Step 1: Asymptotics for the finite sum.}

    Recalling the choice of  $z_{1}$ and  $z_{2}$ in the setting, after  the change of summation index $j\to N-k-j$ the finite sum $T_{M,N}$ in \eqref{ginibresum} can be  rewritten as
    \begin{align}
      T_{M,N} 
      = (z_{1}\bar{z}_{2})^{N - k} \sum_{j=1-k}^{N-k} e^{- \frac{M}{N-k+1} f_{M,N}(j)} e^{i j ( \phi_{1}-\phi_{2})}, \label{G1-sum1}
    \end{align}
    where  $f_{M,N}(t) = f_{N}(t) + t h_{M,N}$
     with
    \begin{align}
      f_{N}(t) =  (N-k+1) \big(\log\Gamma(N - k + 1 - t) + t \psi(N - k + 1)\big), \label{fN}
          \end{align}
          and   \begin{align}
      h_{M,N}= \frac{N-k+1}{\rho \sqrt{MN}}(v_1+v_2).
    \end{align}
    Next, we  will prove that the term of $j=0$ in  the summation of \eqref{G1-sum1}  gives a dominant   contribution and the total contribution of the other terms is negligible as $M+N\to \infty$. With the notation in Lemma \ref{lem-g}, we can rewrite
    \begin{align}
      f_{M,N}(t) - f_{M,N}(0)
      &= g(N - k + 1;t)+ th_{M,N}.
    \end{align}
In order to obtain approximation of the finite sum, we need to  tackle the  two cases that   $N - k$ tends to infinity   and $N - k$ is finite  as $M+N\to \infty$.

  {\bf Case 1}:  $N - k \to \infty$ as $M+N\to \infty$. In this case, use statement \ref{lem-g3} in Lemma \ref{lem-g} we have for sufficiently large $N-k$
    \begin{align}
      f_{N}(t) - f_{N}(0)
      &=\frac{t^{2}}{2}\Big(1 + \bigO\big((N - k)^{-\frac{3}{4}}\big)\Big)\geq \frac{1}{4}t^2
    \end{align}
  whenever    $\abs{t} \leq (N - k)^{\frac{1}{4}}$.  Noting that $|h_{M,N}|\leq  \sqrt{M/(N-k+1)}$ for large $N-k$,
   we obtain
  for $1\leq \abs{t} \leq (N - k)^{\frac{1}{4}}$ that
    \begin{align}
      f_{M,N}(t) - f_{M,N}(0)
       \geq \frac{1}{4}t^2-\sqrt{\frac{N-k+1}{M}}\big|(v_{1}+v_2)t\big|.  \label{gin1-lbound}   \end{align}
  Moreover,  the $t$-function on the right-hand side is increasing over $[1/2,\infty)$ whenever  $\sqrt{(N-k+1)/M} |(v_{1}+v_2)|\leq 1/8$.
   By introducing
  \begin{equation}
  J_{in}=[1 - k, N-k]\cap [-\lfloor(N - k)^{\frac{1}{4}}\rfloor, \lfloor(N - k)^{\frac{1}{4}} \rfloor],
  \end{equation}
  \begin{equation}J_{out}=[1 - k, N-k]\!\setminus\! [-\lfloor (N - k)^{\frac{1}{4}} \rfloor,  \lfloor(N - k)^{\frac{1}{4}} \rfloor],
  \end{equation}
  noting that $v_1, v_2$ are in a compact of $\mathbb{R}$, we thus get
    \begin{align}  \label{ginibre-1-sum-main}
    \begin{split}
      &
     \Big| \sum_{0\neq j\in J_{in}}e^{-\frac{M}{N-k+1} f_{M,N}(j)} e^{i j ( \phi_{1}-\phi_{2})}\Big| 
     \\
      &\leq 2e^{-\frac{M}{N-k+1} f_{M,N}(0)}  \sum_{j=1}^{\lfloor (N - k)^{\frac{1}{4}} \rfloor}
        e^{-\frac{M}{N-k+1}\big(\frac{1}{4}t^2-\sqrt{\frac{N-k+1}{M}}|(v_{1}+v_2)t|\big) } 
        \\
      &\leq 4e^{-\frac{M}{N-k+1} f_{M,N}(0)}  \int_{1/2}^{\infty} dt\,
        e^{-\frac{M}{N-k+1}\big(\frac{1}{4}t^2-\sqrt{\frac{N-k+1}{M}}|(v_{1}+v_2)t|\big) } 
        \\
      &= e^{-\frac{M}{N-k+1} f_{M,N}(0)}  \bigO\big(\sqrt{(N-k+1)/M} \big),
      \end{split}
    \end{align}
    where in the last estimate  use has been made of the change of variables $t \to t\sqrt{(N-k+1)/M}$.
    Also by statement \ref{lem-g3} in Lemma \ref{lem-g}, it immediately follows that
    \begin{align}
      \inf_{t \in  J_{out}} &\Re\{f_{M,N}(t) - f_{M,N}(0)\}       \geq  \frac{1}{8} \sqrt{N - k+1},
    \end{align}
   from which we arrive at   \begin{align} \label{ginibre-1-sum-out}  \begin{split}
      &
     \Big| \sum_{ j\in J_{out}}e^{-\frac{M}{N-k+1} f_{M,N}(j)} e^{i j ( \phi_{1}-\phi_{2})}\Big|\\
      &\leq N e^{-\frac{M}{N-k+1} f_{M,N}(0) - \frac{M}{ 8\sqrt{N - k+1}}} \\
      &\leq N e^{-\frac{M}{N-k+1} f_{M,N}(0) - \frac{M}{ 8\sqrt{N}}}.    \end{split}
    \end{align}
 Combination of \eqref{ginibre-1-sum-main} and \eqref{ginibre-1-sum-out} shows that
    \begin{align}
      T_{M,N} = (z_{1} \bar{z}_{2})^{N - k} e^{-\frac{M}{N-k+1} f_{M,N}(0)}\Big (1 +
        \bigO\big(\sqrt{\frac{N-k+1}{M}} \big)+ \bigO\big(N
       e^{- \frac{M}{ 8\sqrt{N} } }
        \big)
        \Big). \label{ginibre-1-sum1}
    \end{align}
     Here we stress that $N
       e^{- \frac{M}{ 8\sqrt{N} } } \to 0$ when both $M\to \infty$ and  $M/N \to \infty$.

     {\bf Case 2}:   $N - k$ is a fixed number  as $M+N\to \infty$.   In such case Lemma \ref{lem-g} shows  that
     \begin{align}
      \inf_{|t|\geq 1, t \in [1-k,N-k]} &\Re\{f_{M,N}(t) - f_{M,N}(0)\} \geq \frac{1}{2} g(N - k + 1, 1) =: \frac{1}{2}\delta_{min} > 0.
    \end{align}
    As a consequence,
     \begin{align}
\Big| \sum_{0\neq j\in[1-k,N-k]}e^{- \frac{M}{N-k+1} f_{M,N}(j)} e^{i j ( \phi_{1}-\phi_{2})}\Big|\leq e^{- \frac{M}{N-k+1} f_{M,N}(0)} N e^{- \frac{M\delta_{min}}{2(N-k+1)}}.
    \end{align}
   Therefore, we obtain a similar  estimate as in Case 1
     \begin{align}
      T_{M,N} = (z_{1} \bar{z}_{2})^{N - k} e^{-\frac{M}{N-k+1} f_{M,N}(0)}\big (1 +
        \bigO\big(
       N e^{- \frac{M\delta_{min}}{2(N-k+1)}})
        \big). \label{ginibre-1-sum}
    \end{align}

  {\bf Step 2: Asymptotics for the weight function.}

After  the change of variables $s\to s-(N-k+1)$ in   \eqref{ginibreweight}, the weight function can be rewritten   as
    \begin{align}
      w(z_{1})= \abs{z_{1}}^{-2(N - k + 1)} \int_{- i\infty}^{i\infty} \frac{ds}{2 \pi i} e^{\frac{M}{N-k+1} f_{N}(s) + 2 s \sqrt{\frac{M}{N}} \frac{v_1}{\rho}},
    \end{align}
    where $f_{N}(s)$ is defined in \eqref{fN}. Let $s=iy$, statement \ref{lem-g2} in Lemma \ref{lem-g} shows that
     \begin{align}
      \Re\{ f_N(iy)- f_N(0)\}&=\Re{g(N - k + 1, iy)} \nonumber\\
      &\leq - \frac{\abs{y}(N - k + 1)}{4} \min\big(1, \frac{\abs{y}}{N - k + 1}\big).\label{upperbound}
    \end{align}
   This  can easily be used to prove  the estimate
    \begin{align}
      & \Big |\Big(\int_{i \delta}^{i\infty}+ \int_{-i\infty}^{-i \delta} \Big)\frac{ds}{2 \pi i } e^{\frac{M}{N-k+1} (f_{N}(s)-f_{N}(0)) + 2s \sqrt{\frac{M}{N}} \frac{v_1}{\rho}}  \Big|
      \nonumber\\
      &\leq  2 \Big(\int_{ \delta}^{N -  k+1}dy\,  e^{-\frac{M}{4(N - k+1)} y^{2}}  + \int_{N -  k+1}^{\infty}dy\,  e^{-\frac{M}{8}  y}\Big)\nonumber\\
      &\leq
      \Big(  \frac{2(N-k+1)}{M\delta } e^{-\frac{M\delta^2}{4(N - k+1)}  }  +    \frac{8}{M} e^{-\frac{1}{8}M(N-k+1)}\Big), \label{1-weight-1}
    \end{align}
    which
    decays  exponentially to zero.

    Finally, use the standard  steepest descent argument and  we see that the leading contribution for the integral in   the weight function $w(z_{1})$ comes from the integration over some small neighborhood  of the saddle point  $s_0=0$.  Noticing the fact that $f_{N}''(0)=(N-k+1) \psi'(N - k + 1)$ is a fixed positive number ($N-k$ is fixed) or approximates some positive constant ($N-k\to\infty$),  take a Taylor expansion at zero  and  we have
     \begin{align}
     &\int_{- i\delta}^{i\delta} \frac{ds}{2 \pi i} e^{\frac{M}{N-k+1} (f_{N}(s)-f_{N}(0)) + 2 s \sqrt{\frac{M}{N}} \frac{v_1}{\rho}}=\frac{1}{\sqrt{ M \psi'(N - k + 1)})}\nonumber\\
     &\times
     \Big( \int_{-i\infty}^{i\infty} \frac{ds}{2 \pi i} e^{\frac{1}{2} s^{2} + v_1 s }+  \bigO\Big(\frac{1}{\sqrt{ M \psi'(N - k + 1)}}\Big)\Big). \label{1-weight-2}
\end{align}
   Combination of  \eqref{1-weight-1} and \eqref{1-weight-2}
 gives rise to
    \begin{equation}\label{ginibre-1-weight}
      w(z_{1})
      =\frac{\abs{z_{1}}^{-2(N - k + 1)} }{ \sqrt{2\pi M \psi'(N - k + 1)}} e^{\frac{M}{N-k+1} f_{N}(0)-\frac{v_{1}^2}{2}}  \Big(1 +  \bigO\Big(\frac{1}{\sqrt{ M \psi'(N - k + 1)}}\Big)\Big).
    \end{equation}

    Combine  \eqref{ginibre-1-sum1}, \eqref{ginibre-1-sum}  and \eqref{ginibre-1-weight}, we get an  approximation of the correlation  kernel, uniformly for $v_1, v_2$ in a compact subset of $\mathbb{R}$,        \begin{align}
       K_{M, N}(z_{1}, z_{2}) = \frac{1}{(2 \pi)^{\frac{3}{2}}} e^{- \frac{1}{4} (v_{1}^{2} + v_{2}^{2})}  \frac{(z_{1}\bar{z}_{2})^{N-k}}{\abs{z_{1}z_{2}}^{N-k+1}}   \frac{1+o(1)}{\sqrt{ M \psi'(N - k + 1)} },
    \end{align}
from which   Theorem \ref{ginibre-1}   immediately follows.
  \end{proof}

%

\subsection{Proof of Theorem \ref{ginibre-2}}   In this case  both  $M$ and $N$   tend to infinity whenever  $M+N\to \infty$.

\begin{proof} [Proof of Theorem \ref{ginibre-2}]
 We proceed in two steps, in which the weight function and finite sum are discussed  respectively, and then complete  the proof.

    {\bf Step 1: Asymptotics for the finite sum.}

     By Cauchy's  residue theorem,  we rewrite the finite sum  $T_{M,N}$ as
    \begin{align}
      T_{M,N} 
      &= (z_{1} \bar{z}_{2})^{N- \lfloor qN \rfloor - 1} \oint_{\Sigma}
      \frac{dt}{2 \pi i} \frac{(z_{1} \bar{z}_{2})^{-t} (-1)^{t} \pi}{\Gamma^{M}(N - \lfloor qN \rfloor - t) \sin \pi t}, \label{sum-2-int}
    \end{align}
    where $\Sigma$ is an anticlockwise contour  just encircling    $ - \lfloor qN \rfloor, 1- \lfloor qN \rfloor \ldots, N-1- \lfloor qN \rfloor$ and will be given below in great detail.

    For convenience, let
    \begin{equation}\beta=\frac{\gamma}{1-q}.\end{equation}
    With the choice of $z_1,z_2$, use the notation in Lemma \ref{lem-g} and we write
    \begin{align}
      &\log\Big(\frac{\Gamma^{M}(N - \lfloor qN \rfloor) }{\Gamma^{M}(N - \lfloor qN \rfloor - t)} (z_{1} \bar{z}_{2})^{-t}\Big)=
     -
      \frac{M}{N - \lfloor qN \rfloor} g(N - \lfloor qN \rfloor, t)\nonumber
      \\
       & - t M \Big(N - \lfloor qN \rfloor -  \psi(N - \lfloor qN \rfloor)+\log\big(1  + \frac{v_1-\frac{\beta}{4}}{\sqrt{\gamma MN}} \big)\big(1  + \frac{v_2-\frac{\beta}{4}}{\sqrt{\gamma MN}} \big)\Big).\label{ginibre-2-int}
    \end{align}
    By   statement \ref{lem-g3} in Lemma \ref{lem-g} and the approximation of the digamma function
    \cite[5.11.1]{Olver10}
     \begin{align}
      &\psi(z)=\log z -\frac{1}{2z}+\bigO\big(\frac{1}{z^2}\big),  \quad z \to \infty, \label{asy-psi}
    \end{align}
    we have    for  $\abs{t} =\bigO(N^{\frac{1}{4}})$ 
    \begin{align}
      &\log\Big(\frac{\Gamma^{M}(N - \lfloor qN \rfloor) }{\Gamma^{M}(N - \lfloor qN \rfloor - t)} (z_{1} \bar{z}_{2})^{-t}\Big)\nonumber\\
      &=-\frac{M}{2(N - \lfloor qN \rfloor)} t^{2} - \sqrt{\frac{M}{\gamma N}} (v_{1} + \bar{v}_{2} + o(1)) t + \bigO(N^{-\frac{1}{4}}). \label{asy-crit-1}
    \end{align}


    Now let's  specify the contour $\Sigma$ in \eqref{sum-2-int} as a rectangular contour  with four vertices
 $$ -\lfloor qN \rfloor - \frac{1}{2} \pm \frac{i}{8}, \quad N - \lfloor qN \rfloor - \frac{1}{2}\pm \frac{i}{8}.$$  Precisely,   let's define
    \begin{equation}\label{crit-sum-contours-1}
      \begin{gathered}
       \Sigma_{1}^{\pm} = \big\{x \pm \frac{i}{8}: x \in [-\lfloor N^{\frac{1}{4}} \rfloor + \frac{1}{2}, \lfloor N^{\frac{1}{4}} \rfloor - \frac{1}{2}]\big\},\\
        \Sigma_{2}^{\pm} = \big\{x \pm \frac{i}{8}: x \in [-\lfloor qN \rfloor - \frac{1}{2}, -\lfloor N^{\frac{1}{4}} \rfloor + \frac{1}{2}]\big\}, \\
          \Sigma_{12}^{\pm} = \big\{x \pm \frac{i}{8}: x \in [ - \frac{1}{2}, \lfloor N^{\frac{1}{4}} \rfloor - \frac{1}{2}]\big\}, \\
        \Sigma_{3} = \big\{-\lfloor qN \rfloor - \frac{1}{2} + iy: y \in [-\frac{1}{8}, \frac{1}{8}]\big\},\\
         \Sigma_{4} = \big\{N - \lfloor qN \rfloor - \frac{1}{2} + iy: y \in [-\frac{1}{8}, \frac{1}{8}]\big\},\\
        \Sigma_{5}^{\pm} = \big\{x \pm \frac{i}{8}: x \in [\lfloor N^{\frac{1}{4}} \rfloor - \frac{1}{2}, N - \lfloor qN \rfloor - \frac{1}{2}]\big\},
      \end{gathered}
    \end{equation}
 we  then choose
  \begin{equation}\label{crit-sum-contours-bulk}
    \Sigma= \Sigma_{1}^{+} \cup  \Sigma_{2}^{+}\cup  \Sigma_{3}\cup  \Sigma_{2}^{-}\cup  \Sigma_{1}^{-}\cup  \Sigma_{5}^{-}\cup  \Sigma_4 \cup \Sigma_{5}^{+}
  \end{equation}
  if $q \in (0, 1)$, while
  \begin{equation}\label{crit-sum-contours-edge}
    \Sigma= \Sigma_{12}^{+} \cup    \Sigma_{3}\cup   \Sigma_{12}^{-}\cup  \Sigma_{5}^{-}\cup  \Sigma_4 \cup \Sigma_{5}^{+}
  \end{equation}
  if $q=0$.
So as $M+N\to \infty$ and $M/N\to \gamma$, when  $q \in (0, 1)$, we have from \eqref{asy-crit-1} that

    \begin{align}   
      &\int_{\Sigma_{1}^{+} \cup   \Sigma_{1}^{-} } \frac{dt}{2 \pi i} \frac{\Gamma^{M}(N - \lfloor qN \rfloor) (z_{1} \bar{z}_{2})^{-t} (-1)^{t}\pi}{\Gamma^{M}(N - \lfloor qN \rfloor - t) \sin \pi t}\nonumber\\
      &\to \int_{\mathds{R}\pm \frac{i}{8}} \frac{dt}{2 \pi i} \frac{(-1)^{t}\pi}{\sin \pi t} e^{- \frac{1}{2} \beta t^{2} - t (v_1+\bar{v}_2)}\nonumber\\
     &=   \sum_{j = - \infty}^{\infty} e^{- \frac{1}{2} \beta  j^{2} - j(v_1+\bar{v}_2) },
    \end{align} uniformly for $v_1,v_2$ in a compact set of $\mathbb{C}$.
   Here the Cauchy's residue theorem has been used to obtain the equality. Similarly,  when  $q = 0$  we have
    \begin{align}
      &\int_{\Sigma_{12}^{+} \cup \Sigma_{3}\cup   \Sigma_{12}^{-} } \frac{dt}{2 \pi i} \frac{\Gamma^{M}(N - \lfloor qN \rfloor) (z_{1} \bar{z}_{2})^{-t} (-1)^{t}\pi}{\Gamma^{M}(N - \lfloor qN \rfloor - t) \sin \pi t}\nonumber\\
      &\to   \int_{\Sigma_{3}\cup \{x\pm \frac{i}{8}: x \in [-\frac{1}{2},\infty)\}} \frac{dt}{2 \pi i} \frac{(-1)^{t}\pi}{\sin \pi t} e^{- \frac{1}{2} \beta t^{2} - t (v_1+\bar{v}_2)}\nonumber\\
     &=   \sum_{j = 0}^{\infty} e^{- \frac{1}{2} \beta  j^{2} - j(v_1+\bar{v}_2) },
    \end{align}
     uniformly for $v_1,v_2$ in a compact set of $\mathbb{C}$.

   The remaining task is to prove that  the remainder of the integral in  \eqref{sum-2-int} is negligible. For $q\in (0,1)$, using statement \ref{lem-g3} in Lemma \ref{lem-g}, we have from \eqref{ginibre-2-int} that
    \begin{align}
      \max_{t \in \Sigma_{2}^{+}\cup  \Sigma_{3}\cup  \Sigma_{2}^{-}\cup  \Sigma_{5}^{-}\cup  \Sigma_4 \cup \Sigma_{5}^{+}}\abs{\frac{\Gamma^{M}(N - \lfloor qN \rfloor) (z_{1} \bar{z}_{2})^{-t}}{\Gamma^{M}(N - \lfloor qN \rfloor - t)}}
      \leq e^{-\frac{1}{4} \beta \sqrt{N}}.
    \end{align}
    Note that $|(-1)^t /\sin\pi t|$ has an upper bound independent of $M,N$ whenever $t$ belongs to the chosen contour, as a consequence, we obtain
    \begin{align}
      \abs{\int_{
     \Sigma_{2}^{+}\cup  \Sigma_{3}\cup  \Sigma_{2}^{-}\cup  \Sigma_{5}^{-}\cup  \Sigma_4 \cup \Sigma_{5}^{+}
      } \frac{dt}{2 \pi i} \frac{\Gamma^{M}(N - \lfloor qN \rfloor) (z_{1} \bar{z}_{2})^{-t} (-1)^{t}}{\Gamma^{M}(N - \lfloor qN \rfloor - t) \sin \pi t}} \leq e^{- \frac{1}{8} \beta \sqrt{N}}
    \end{align}
    when $M$ and $N$ are  sufficiently large.
Similarly, for $q = 0$ we have
    \begin{align}
      \abs{\int_{
       \Sigma_{5}^{-}\cup  \Sigma_4 \cup \Sigma_{5}^{+}
      } \frac{dt}{2 \pi i} \frac{\Gamma^{M}(N - \lfloor qN \rfloor) (z_{1} \bar{z}_{2})^{-t} (-1)^{t}}{\Gamma^{M}(N - \lfloor qN \rfloor - t) \sin \pi t}} \leq e^{- \frac{1}{8} \beta \sqrt{N}}
    \end{align}
    when $M, N$ are  sufficiently large.

 In short,  as $N\to \infty$  we have  uniformly for $v_1,v_2$ in a compact set of $\mathbb{C}$
    \begin{align}\label{asy-crit-1-sum}
      &(z_{1} \bar{z}_{2})^{\lfloor qN \rfloor-N+1} \Gamma^{M}(N - \lfloor qN \rfloor) T_{M,N}\nonumber \\
      &\longrightarrow  \begin{cases}
        \sum_{j = - \infty}^{\infty} e^{- \frac{1}{2} \beta  j^{2} - j(v_1+\bar{v}_2) }, &  q \in (0, 1),\\
     \sum_{j = 0}^{\infty} e^{- \frac{1}{2} \beta  j^{2} - j(v_1+\bar{v}_2) }, & q = 0.
      \end{cases}
    \end{align}

 {\bf Step 2: Asymptotics for the weight function.}

 Changing  $s\to -s- (N - \lfloor qN \rfloor)$
  in the weight function $w(z_{1})$ gives
   \begin{align}
      w(z_{1})= \abs{z_{1}}^{-2(N - \lfloor qN \rfloor)} \int_{- i\infty}^{i\infty} \frac{ds}{2 \pi i}
      \Gamma^{M}(N - \lfloor qN \rfloor - s) |z_1|^{2s},
    \end{align}
    Let $s=iy$, we can derive a similar result to  \eqref{upperbound}
    \begin{align}
      &\abs{\frac{\Gamma^{M}(N - \lfloor qN \rfloor -s)}{\Gamma^{M}(N - \lfloor qN \rfloor)} \abs{z_{1}}^{2s}}
      \leq \exp\Big\{-\frac{M\abs{y}}{4} \min\Big(1,\frac{\abs{y}}{N - \lfloor qN \rfloor}\Big)\Big\}
    \end{align}
    from which
    \begin{align}
      &\abs{\Big(\int_{- i \infty}^{-i N^{\frac{1}{4}}} + \int_{i N^{\frac{1}{4}}}^{i \infty}\Big) \frac{ds}{2 \pi i}
     \frac{\Gamma^{M}(N - \lfloor qN \rfloor -s)}{\Gamma^{M}(N - \lfloor qN \rfloor)}  \abs{z_{1}}^{2 s}}\nonumber\\
      &\leq 2  \bigg(\int_{N^{\frac{1}{4}}}^{N - \lfloor qN \rfloor}dy  e^{-\frac{M}{4(N - \lfloor qN \rfloor)} y^{2}}  + \int_{N - \lfloor qN \rfloor}^{\infty}dy  e^{-\frac{M}{8}  y}\bigg)\nonumber\\
      &\leq 2  \bigg(N^{-\frac{1}{4}}e^{-\frac{M\sqrt{N}}{4(N - \lfloor qN \rfloor)}  }  +    \frac{8}{M} e^{-\frac{1}{8}M(N - \lfloor qN \rfloor)}\bigg). \label{2-weight-1}
    \end{align}
    When $|s|\leq N^{\frac{1}{4}}$, we can proceed as in \eqref{asy-crit-1}, take a Taylor expansion and get
    \begin{align}
      &\int_{- i N^{\frac{1}{4}}}^{i N^{\frac{1}{4}}} \frac{ds}{2 \pi i} \frac{\Gamma^{M}(N - \lfloor qN \rfloor - s)}{\Gamma^{M}(N - \lfloor qN \rfloor)} \abs{z_{1}}^{2 s} \nonumber\\
      &\longrightarrow   \int_{- i \infty}^{i \infty} \frac{ds}{2 \pi i} e^{\frac{1}{2} \beta s^{2} + s(v_1+\bar{v}_1)}=\frac{1}{\sqrt{2\pi \beta}} e^{-\frac{1}{2\beta}(v_1+\bar{v}_1)^2}.\label{2-weight-2}
    \end{align}
    Together with \eqref{2-weight-1}, we have
      \begin{align}
      &\frac{\abs{z_{1}}^{2(N - \lfloor qN \rfloor)}}{\Gamma^{M}(N - \lfloor qN \rfloor)}  w(z_1)
     \longrightarrow   \frac{1}{\sqrt{2\pi \beta}} e^{-\frac{1}{2\beta}(v_1+\bar{v}_1)^2}.\label{2-weight-3}
    \end{align}

  Combining \eqref{asy-crit-1-sum} and \eqref{2-weight-3}, we get an  approximation of the correlation  kernel, uniformly for $v_1, v_2$ in a compact subset of $\mathbb{C}$,        \begin{align}
       K_{M, N}(z_{1}, z_{2}) = \frac{1}{\sqrt{2\pi^3 \beta}}    \frac{(z_{1}\bar{z}_{2})^{N-\lfloor qN \rfloor-1}}{\abs{z_{1}z_{2}}^{N-\lfloor qN \rfloor}}  e^{\frac{1}{4\beta}((v_{2}^2-\bar{v}_{2}^2)-(v_{1}^2-\bar{v}_{1}^2))}  (1+o(1)) \nonumber\\
       \times  \begin{cases}
         \mathrm{K^{(bulk)}_{\mathrm{crit}}}(\frac{\gamma}{1-q};v_1,v_2), & q \in (0, 1),\\
     \mathrm{K^{(edge)}_{\mathrm{crit}}}(\gamma;v_1,v_2), & q = 0,
      \end{cases}
    \end{align}
from which   Theorem \ref{ginibre-2}   immediately follows.
  \end{proof}

\subsection{Proof of Theorem \ref{ginibre-3}} When $M$ is finite, the local statistics for Model A  in the bulk and at the soft edge has been proved to be the same as in a single Ginibre matrix, see \cite{AB12,LiuWang16}.  In order to tackle  the difficulty  of how to obtain asymptotics for  the truncated series as $N\to \infty$,
we need to introduce  a different method  with those used   in the  proofs of Theorems  \ref{ginibre-1} and  \ref{ginibre-2}, which  works well  for any integer  $M$ such that $M/N\to 0$ as $N\to \infty$.  The  key point  is  to combine the  Euler-Maclaurin summation formula (see e.g. \cite[Theorem 1]{Lampret}) and the  steepest  decent  argument.

 First, we  state a general result  that is used to deal with   asymptotics  for a finite sum of truncated series depending on two parameters $M$ and $N$.
  \begin{lem}\label{lem-eulmac}
   Given  two non-negative integers $M$ and $N$,   for a parameter $\xi\in \mathbb{C}$ let $f_{M,N}(t):=f_{M,N}(\xi,t)$ be a  complex-valued function of $t$  over $[0, \infty)$ such that  the following conditions are satisfied:
    \begin{enumerate}[label = (A\arabic*)]
      \item\label{lem-eulmac-st-1} There exists some constant $\delta \in (0, 1)$ such that for large $N$ $\Re{f_{M,N}(t)}$ attains  its minimum over  $[0, \delta N]$  at $\delta N$.
      \item\label{lem-eulmac-st0} For   $t \in [\delta/2, 1]$,   $f_{M,N}(tN)$ can be rewritten as            \begin{align}
            f_{M,N}(tN) &= MN f(t) +  \sqrt{MN}\xi t  + Q_{M,N}(\xi, t), \quad \label{lem-eulmac-expansion}
          \end{align}
        such that  as $N\to \infty$ and $M/N\to 0$,        $ |Q_{M,N}(\xi,t)| = \bigO(\alpha_{M,N})$ and   $ |Q'_{M,N}(\xi,t)| =\bigO( \alpha_{M,N})$ with  a positive sequence  $\alpha_{M,N}= o(\sqrt{MN})$,
        uniformly for  $t \in [\delta/2, 1]$ and $\xi$ in  any compact set  of $\mathbb{C}$. 

      \item\label{lem-eulmac-st1} $f(t)$ is a real-valued continuous function and $f'(t), f''(t),f'''(t)$ are  continuous over $[\delta/2, 1]$;  $f(t)$ has a unique minimum at $t_{0} \in (\delta, 1]$ such that  $f'(t_{0}) = 0$ and $f''(t_{0}) > 0$.
    \end{enumerate}
   If  $ M/N \to 0$   as $M+N \to \infty$, then  we have
    \begin{align} \label{sumbasic}
      \sum_{j=0}^{N-1} e^{-f_{M,N}(j)} &= e^{-f_{M,N}(t_{0}N)} \frac{\sqrt{2\pi N}}{\sqrt{Mf''(t_{0})}} e^{- \frac{1}{2 f''(t_{0})} \xi^{2}} 
    \nonumber  \\
      &\quad \times \begin{cases}
       \big(1 + \bigO\big( \frac{ 1}{\sqrt{MN}}\big)\big), &  t_{0} \in (0, 1),\\
        \frac{1}{2} \mathrm{erfc}(\frac{-\xi}{\sqrt{2 f''(t_{0})}})
        \big (1 + \bigO\big(  \frac{1}{\sqrt{MN}} \big)\big), &  t_{0}=1,
      \end{cases}
    \end{align}
    uniformly for  $\xi$ in  a compact set  of $\mathbb{C}$.
  \end{lem}
  \begin{proof}
  Noting $\delta \in (0, t_{0})$,   let's divide the sum on the left-hand side of \eqref{sumbasic}   into two parts according to the index  $j< \lfloor \delta N \rfloor$ or $j\geq\lfloor \delta N \rfloor$,  apply the Euler-Maclaurin summation formula (see e.g.  \cite[Theorem 1]{Lampret}) 
   to the second part and  we  obtain
    \begin{align}
      \sum_{j=\lfloor \delta N \rfloor}^{N - 1} e^{-f_{M,N}(j)}
      &= \int_{\lfloor \delta N \rfloor}^{N} e^{-f_{M,N}(t)} dt + \int_{\lfloor \delta N \rfloor}^{N} \big(\lfloor t \rfloor - t + \frac{1}{2}\big) f_{M,N}'(t) e^{-f_{M,N}(t)} dt \nonumber\\
      &\quad + \frac{1}{2} \big(e^{-f_{M,N}(\lfloor \delta N \rfloor)} - e^{-f_{M,N}(N)}\big)\nonumber\\
      &=: I_{1} + I_{2} + I_{3}.
    \end{align}

    Using the change of variables  $t \to t N$ and   by the assumption \ref{lem-eulmac-st0},  we divide $I_1$ into  two parts  \begin{align}
      I_{1} &
      = \Big(\int_{t_{0} - \epsilon}^{\min\{t_{0}+\epsilon,1\}} dt + \int_{[\frac{1}{N}\lfloor \delta N \rfloor, 1] \backslash (t_{0}-\epsilon, t_{0} + \epsilon)} dt\Big)\, N e^{-MNf(t) - \sqrt{MN}  \xi t-  Q_{M,N}(\xi, t)}\nonumber\\
      &=:I_{11} + I_{12} \label{eulmac-i10}
    \end{align}
    By the assumption \ref{lem-eulmac-st1}, for any small $\epsilon > 0$  there exists $C_\epsilon>0$ such that
    \begin{align}
      \inf_{t \in [\lfloor \delta N \rfloor/N, 1] \backslash (t_{0} - \epsilon, t_{0} + \epsilon)}(f(t) - f(t_{0})) \geq C_\epsilon.    \end{align}
  Together with  the assumption on $Q_{M,N}(\xi,t)$, we have for  $N$ sufficiently  large
    \begin{align}
      \abs{I_{12}} \leq  e^{-f_{M,N}(t_{0}N)}   e^{-\frac{1}{2}C_\epsilon MN}.\label{eulmac-i12}
    \end{align}
   For $I_{11}$,  taking  the Taylor expansion of $f(t)$ at   $t_{0}$  and using the assumptions  on $Q_{M,N}(\xi,t), Q'_{M,N}(\xi,t)$,   we have
    \begin{align}
      I_{11}&= N  e^{-f_{M,N}(t_{0}N)}  \int_{t_{0} - \epsilon}^{\min\{t_{0}+\epsilon,1\}} dt\,
      e^{- \frac{1}{2} MNf''(t_{0}) (t-t_{0})^{2} }e^{- \sqrt{MN} \xi (t - t_{0})  }
      \nonumber \\
      &\quad\times  e^{  \bigO(MN)(t-t_{0})^{3} -    \bigO(\alpha_{M,N})(t-t_{0})}.
    \end{align}
    Use the change of variables  $t \to t_0+t/\sqrt{MN}$ and we   see from the standard   steepest  decent  argument that
       \begin{align}
      I_{11}&=  e^{-f_{M,N}(t_{0}N)} \sqrt{\frac{N}{M}}
 \int_{- \epsilon \sqrt{MN} }^{\min\{\epsilon,1-t_{0}\}\sqrt{MN} } e^{- \frac{1}{2}  f''(t_{0}) t^{2} -  \xi t + \bigO( \frac{1}{\sqrt{MN}}) t^{3} + \bigO( \frac{\alpha_{M,N}}{\sqrt{MN}} )t}\nonumber\\
      &= \frac{\sqrt{2\pi N}}{\sqrt{Mf''(t_{0})}} e^{-f_{M,N}(t_{0}N)}  e^{- \frac{1}{2 f''(t_{0})}\xi^{2}} \nonumber\\
      &\quad \times \begin{cases}
        1 + \bigO\big(\frac{\alpha_{M,N}}{\sqrt{MN}}\big)+\bigO\big(\frac{1}{\sqrt{MN}}\big), &  t_{0} \in (0, 1),\\
        \frac{1}{2} \mathrm{erfc}(\frac{-\xi}{\sqrt{2 f''(t_{0})}})
        \big(1 + \bigO\big(\frac{\alpha_{M,N}}{\sqrt{MN}}\big)+\bigO\big(\frac{1}{\sqrt{MN}}\big)\big), & t_{0}=1.
      \end{cases}\label{eulmac-i11}
    \end{align}
    Thus, combine \eqref{eulmac-i12} and   \eqref{eulmac-i11} and   we obtain
     \begin{align}
      I_{1}      &= \frac{\sqrt{2\pi N}}{\sqrt{Mf''(t_{0})}} e^{-f_{M,N}(t_{0}N)}  e^{- \frac{1}{2 f''(t_{0})}\xi^{2}} \nonumber\\
      &\quad \times \begin{cases}
        1 + \bigO\big(\frac{\alpha_{M,N}}{\sqrt{MN}}\big)+\bigO\big(\frac{1}{\sqrt{MN}}\big), &  t_{0} \in (0, 1),\\
        \frac{1}{2} \mathrm{erfc}(\frac{-\xi}{\sqrt{2 f''(t_{0})}})
        \big(1 + \bigO\big(\frac{\alpha_{M,N}}{\sqrt{MN}}\big)+\bigO\big(\frac{1}{\sqrt{MN}}\big)\big), &  t_{0}=1,
      \end{cases}\label{eulmac-i1}
    \end{align}
uniformly for $\xi$ in  a compact set  of $\mathbb{C}$.

Similarly, in order to estimate $I_2$, we    divide  the integral $I_2$ into two corresponding parts  like in $I_1$ and write $I_2=I_{21}+I_{22}$. Note that  by the assumption as   $N\to \infty$ and $M/N\to 0$
    \begin{align}
      f_{M,N}'(tN) = Mf'(t) + \bigO(\sqrt{M/N})
    \end{align}
     uniformly for all $t \in [\delta/2, 1]$.  Compare the integrals $I_{22}$ and $I_{12}$,  
      we easily get
    \begin{align}
      \abs{I_{22}} \leq \bigO(M) \abs{I_{12}}.
    \end{align}
   Compare  $I_{22}$ and $I_{12}$, note $f'(t)=\bigO(|t-t_0|)$ when $|t-t_0|<\epsilon$, we
 get
    \begin{align}
      \abs{I_{21}} \leq \bigO(\sqrt{M/N}) \abs{I_{11}}.
    \end{align}
   Furthermore,  put two parts together and    we know
    \begin{align}
      \abs{I_{2}} \leq \bigO(\sqrt{M/N}) \abs{I_{1}}. \label{eulmac-i2}
    \end{align}

    For $I_3$,  it is easy to derive from  the asymptotics  of  $ f_{M,N}(\lfloor \delta N \rfloor )$ and   $ f_{M,N}(N)$ (cf. \eqref{eulmac-i10}) that
    \begin{align}
      |e^{-f_{M,N}(\lfloor \delta N \rfloor )}|\leq \bigO(1) |I_{12}|  \label{I3-1}
    \end{align}
    and
    \begin{align}
      |e^{-f_{M,N}( N)}|
      \leq \begin{cases}\bigO(1) |I_{12}|, & t_0\in (0,1),\\
        \bigO(\sqrt{M/N} ) |I_{11}|, & t_0=1,
      \end{cases}
    \end{align}
    from which  we know
    \begin{align}
      |I_{3}|\leq &\bigO\big(\sqrt{ M/N}\big) |I_{1}|. \label{eulmac-i3}
    \end{align}

    Finally,   using of   \eqref{I3-1} and the assumption  \ref {lem-eulmac-st-1}  give us      \begin{align}
      \Big|\sum_{j=0}^{\lfloor \delta N \rfloor - 1} e^{-f_{M,N}(j)}\Big|
      \leq \delta N e^{-\Re{f_{M,N}(\delta N)}} \leq \bigO(N) |I_{12}|
      \leq \bigO\big(\sqrt{ M/N}\big) |I_{1}|. \label{eulmac-i0}
    \end{align}

    Combining \eqref{eulmac-i1}, \eqref{eulmac-i2}, \eqref{eulmac-i3} and \eqref{eulmac-i0},   we obtain the desired result.
  \end{proof}

  \begin{proof}[Proof of Theorem \ref{ginibre-3}]  We calculate asymptotics  for the weight function and the finite sum of the correlation kernel.     For convenience,  let \begin{equation}\varphi_j=1+\frac{v_j}{u\sqrt{MN}}, \quad j=1,2.\end{equation}

    {\bf Step 1: Asymptotics for the weight function.}

  Use a  change of variables $s\to -sN$ in   \eqref{ginibreweight} and the stirling formula, let $s_{0} = \abs{u}^{2}>0$ and the weight function is rewritten   as
    \begin{align}
      w(z_{1})
      = N \Big(\frac{2 \pi}{s_{0}N}\Big)^{\frac{M}{2}}  e^{MN f(s_0) - s_0 M N \log \abs{ \varphi_j}^{2}} J,  \label{3-weight-0}
    \end{align} where  for every  $\delta_0>0$
       \begin{align}
  J &=
       \int_{s_0 -i \infty}^{s_0 + i \infty} \frac{ds}{2 \pi i} e^{MN( f(s)-f(s_0)) - (s-s_0) M N \log \abs{ \varphi_j}^{2} - \frac{M}{2} \log \frac{s}{s_0} + \bigO(\frac{M}{s N})}\\
       &=  \Big(\int_{s_0 -i \delta_0}^{s_0 + i \delta_0} \frac{ds}{2 \pi i} +\int_{s_0 - i \infty}^{s_0 -i \delta_0} \frac{ds}{2 \pi i}+ \int_{s_0 +i \delta_0}^{s_0 + i \infty} \frac{ds}{2 \pi i} \Big)\Big(\cdot\Big)=:J_{0}+J_{-}+J_{+},
    \end{align}
    with
    \begin{align}
      f(s) = s \log s - s - s \log\abs{u}^{2}. \label{3-f}
    \end{align}
    It is easy to see  that  $f'(s)=\log s- \log\abs{u}^{2}=0$ has a unique solution $s=s_0$ and  with $s=s_{0}+iy$
     \begin{align}
        \frac{\partial}{\partial y} \Re{f (s_0 + iy)}&=-\Im f'(s_0 + iy)=-\arctan\frac{y}{s_0}\nonumber\\
      & \begin{cases}\leq - \arctan\frac{\delta_0}{s_0}, &y\geq \delta_0\\
     \geq  \arctan\frac{\delta_0}{s_0}, &y\leq -\delta_0\end{cases}.
    \end{align}
    So $\Re{f (s_0 + iy)}$ with $y\in \mathds{R}$ attains  its  unique  maximum $y=0$, together with the standard steepest decent argument (see e.g. \cite{Wong}) implying that for sufficiently large $N$
          \begin{align}
 |J_{\pm}|\leq  e^{\frac{1}{2}MN\Re \{f(s_0\pm i\delta_0)-f(s_0)\}}. \label{3-weight-1}
    \end{align}

    Choose a sufficiently small $\delta_0>0$ and  take a Taylor expansion at $s=s_0$, noting $f''(s_{0}) = \abs{u}^{-2}$ and the assumption of  $M/N\to 0$, we have
        \begin{align}
  J_0 &=  \frac{1}{ \sqrt{MN}}\int_{-i \delta_0 \sqrt{MN}}^{ i \delta_0 \sqrt{MN}} \frac{ds}{2 \pi i}
        e^{\frac{1}{2}f''(s_0)s^2 - s  (\frac{v_1}{u}+\frac{\bar{v}_1}{\bar{u}})} \Big(1+ \bigO\Big(\sqrt{\frac{M}{N}}\Big)\Big)\nonumber\\
        &=\frac{|u|}{ \sqrt{2\pi MN} } e^{-\frac{1}{2} (v_1 e^{-i \theta} + \bar{v}_1 e^{i \theta})^{2}} \Big(1+ \bigO\Big(\sqrt{\frac{M}{N}}\Big)\Big). \label{3-weight-2}
    \end{align}

     Combination of  \eqref{3-weight-0}, \eqref{3-weight-1} and \eqref{3-weight-2}  gives rise to
    \begin{align}
      w(z_{1}) &=
       e^{MN f(|u^2|)}  \abs{ \varphi_1}^{-2MN|u^2|} e^{-\frac{1}{2} (v_1 e^{-i \theta} + \bar{v}_1 e^{i \theta})^{2}} \nonumber \\
       & \times  \Big(\frac{2 \pi}{N|u^2|}\Big)^{\frac{M}{2}} \sqrt{\frac{N|u^2|}{2\pi M}}\Big(1+ \bigO\Big(\sqrt{\frac{M}{N}}\Big)\Big), \label{weightest}
    \end{align}
    uniformly for $v_1,v_2$ in a compact set of $\mathds{C}$.

    {\bf Step 2: Asymptotics for the finite sum.}

    Write
    \begin{align}
      T_{M,N} = \sum_{j=0}^{N-1} e^{-f_{M,N}(j)},
    \end{align}
    where
    \begin{align}
      f_{M,N}(t) = M \big( \log\Gamma(t+1)-t \log \big(N\abs{u}^{2} \big)- t\log \varphi_1 \bar{\varphi}_2    \big).
    \end{align}
    Let $\delta\in (0,|u|^2)$. Note that for $t \in (0, \delta N)$
    \begin{align}
      \Re f_{M,N}'(t)
      &= M\big(\psi(t+1) - \log (N\abs{u}^{2} )- \log |\varphi_1 \bar{\varphi}_2| \big)<0
    \end{align}
    when $N$ is large. This shows that $\Re f_{M,N}(t)$ obtains its minimum at $\delta N$ for $t \in [0, \delta N]$,  so the condition \ref{lem-eulmac-st-1} is satisfied.
Meanwhile, applying the Stirling formula we obtain
    \begin{align}
      f_{M,N}(tN) &= MNf(t) - t {MN} \log(\varphi_1 \bar{\varphi}_2) + \frac{M}{2} \log t + \bigO(MN^{-1})\\
      f_{M,N}'(tN) &= M(f'(t) + \bigO((MN)^{-\frac{1}{2}}))
    \end{align}
    hold uniformly for $t \in [\delta/2,1]$, which implies the condition \ref{lem-eulmac-st0}.

    Recalling   $f(t)$ in \eqref{3-f}, it is easy to see from  $f'(t)=\log t- \log |u|^2$ that $f'(s_0)=0$ and $t=s_0$ is a unique minimum point over $(\delta,1]$. It is obvious the $f''(s_{0}) = \abs{u}^{-2} > 0$. Thus \ref{lem-eulmac-st1} holds.

    Now, applying Lemma \ref{lem-eulmac} we get
    \begin{align}
      T_{M,N}&= N  \big(2 \pi N s_0\big)^{-\frac{M}{2}}  e^{-MN f(s_0)}  ( \varphi_1 \bar{\varphi}_2 )^{MNs_0} \frac{|u|}{ \sqrt{MN}} \Big(1+ \bigO\Big(\sqrt{\frac{M}{N}}\Big)\Big)\nonumber\\
      &\times  \sqrt{2\pi}  e^{\frac{1}{2} (v_1 e^{-i \theta} + \bar{v}_2 e^{i \theta})^{2}}
      \begin{cases}
        1, & |u|\in (0,1),\\
        \frac{1}{2} \mathrm{erfc}(\frac{v_1 e^{-i \theta} + \bar{v}_2 e^{i \theta}}{\sqrt{2}}), & |u|=1,
      \end{cases} \label{totalsum}
    \end{align}
    uniformly for $v_1,v_2$ in a compact set of $\mathbb{C}$ as $M + N$ tends to infinity.

    Putting \eqref{weightest} and  \eqref{totalsum}  together,   we obtain
    \begin{align}
      K_{M,N}(z_1,z_2)&=  \big(N |u^2|\big)^{-M}  \Big( \frac{ z_1 \bar{z}_2 }{|z_1 \bar{z}_2|}\Big)^{MN|u^2|} \frac{N|u^2|}{ M} \Big(1+ \bigO\Big(\sqrt{\frac{M}{N}}\Big)\Big)\nonumber\\
      &\times  \frac{1}{\pi}  e^{-\frac{1}{2} (|v_1|^2 +|v_2|^2 - v_1 \bar{v}_2) }\begin{cases}1, & |u|\in (0,1),\\
      \frac{1}{2} \mathrm{erfc}(\frac{v_1 e^{-i \theta} + \bar{v}_2 e^{i \theta}}{\sqrt{2}}), & |u|=1, \end{cases}
    \end{align}
    uniformly for $v_1,v_2$ in a compact set of $\mathbb{C}$. The desired result immediately follows.
  \end{proof}

\section{Products of truncated unitary matrices} \label{sect3}

  Let $U$ be chosen at random  from  the  $(N + L) \times (N + L)$  unitary group with the Haar measure,    the  upper left $N\times N$ corner  $T$  is 
  is  referred as a truncated unitary matrix, which was  first  introduced by \.{Z}yczkowski and Sommers \cite {zyczkowski2000truncations}.
   Our  second product model  is
  \begin{quote}
{\bf Model B}:  $\Pi_M=X_{M}\cdots X_{1}$   where  $X_{1}, \dotsc, X_{M}$ are  i.i.d.  truncated unitary matrices, chosen  from  the  upper left $N\times N$ corner  of a random   unitary  $(N + L) \times (N + L)$ matrix.  \end{quote}
 The eigenvalues of  $\Pi_M$   form a determinantal point process in the complex plane  with kernel \cite{adhikari2013determinantal,ABKN14}
  \begin{align}\label{trunckernel}
    K_{N}(z_{1}, z_{2}) = \frac{1}{\pi} \sqrt{w(z_{1}) w(z_{2})} T_{M,N}(z_{1}, z_{2}), \qquad z_{1}, z_{2} \in \mathbb{C},
  \end{align}
  where
  \begin{align}
    w(z)
    =  \int_{-c - i \infty}^{-c + i \infty}  \frac{ds}{2 \pi i} \frac{\Gamma^{M}(-s)}{\Gamma^{M}(L - s)} \abs{z}^{2s}\label{truncweight}
  \end{align}
  and
  \begin{align}
    T_{M, N} = \sum_{j=0}^{N-1} \frac{\Gamma^{M}(L + j +1)}{\Gamma^{M}(j + 1)} (z_{1} \bar{z}_{2})^{j}.\label{truncsum}
  \end{align}

For a single truncated unitary matrix, one needs to distinguish between  the weakly non-unitary limit  that $L$ is  fixed and the strong non-unitary  limit that  $L$ increases   proportionally with $N$  when  $N\to\infty$; see \cite {khoruzhenko2011, zyczkowski2000truncations}.
 In this section we only consider the  strong non-unitary case.

The key results on products of i.i.d. truncated unitary matrices are stated in the following three theorems.

  \begin{thm}\label{trunc-1}
    Assume  that $\lim_{M+N \to \infty} M/N = \infty$ and $\lim_{M+N \to \infty} L/N = \tau\in (0,\infty)$. Given $k \in \{1, \ldots, N\}$, let $\rho = 2/\sqrt{N(\psi'(N - k +1) - \psi'(L + N - k +1))}$ and  for $\theta \in (- \pi, \pi]$ set
    \begin{align}
      z_{j} &= e^{M(\frac{1}{2} (\psi(L + N - k + 1) - \psi(N - k + 1)) + \frac{v_{j}}{\rho \sqrt{MN}})} e^{i (\theta + \phi_{j})},\quad j = 1, \ldots, n.
    \end{align}
    Then limiting $n$-correlation functions for eigenvalues of $\Pi_M$ in Model B
    \begin{align}
      \lim_{M+N \to \infty}
        \Big(\frac{\sqrt{M} }{\rho\sqrt{ N}}\Big)^{n}  \abs{z_{1}\cdots z_{n}}^2 R^{(n)}_{M,N}(z_{1}, \ldots,z_{n})= \begin{cases}
          \frac{1}{(\sqrt{2\pi})^3}
          e^{-  \frac{1}{2} v_{1}^{2} }, &n=1,\\
          0, & n>1,
        \end{cases}
    \end{align}
    holds uniformly for $v_1, \ldots,v_n$ in a compact subset of $\mathds{R}$ and $\phi_1, \ldots,\phi_n\in (-\pi,\pi]$.
  \end{thm}
  \begin{proof}

    {\bf Step 1: Asymptotics for the finite sum.}

    Recalling the choice of  $z_{1}$ and  $z_{2}$ in this setting,   change the  summation index $j$  to $N-k-j$  and rewrite the finite sum $T_{M,N}$ in \eqref{truncsum}  as
    \begin{align}
      T_{M,N}
      = (z_{1}\bar{z}_{2})^{N - k} \sum_{j=1-k}^{N-k} e^{- \frac{M}{N-k+1} f_{M,N}(j)} e^{i j ( \phi_{1}-\phi_{2})}, \label{trunc-1-sum}
    \end{align}
    where  $f_{M,N}(t) = f_{N}(t) + t h_{M,N}$ with
    \begin{equation}\label{trunc-1-fN}
      \begin{split}
        f_{N}(t) &= (N - k + 1) \big(\log\Gamma(N - k - t + 1) + t\psi(N - k + 1)\big) \\
        &\quad- (N - k + 1) \big(\log\Gamma(L + N - k - t + 1) + t\psi(L + N - k + 1)\big),
      \end{split}
    \end{equation}
    and
    \begin{align}
      h_{M,N}= \frac{N-k+1}{\rho \sqrt{MN}}(v_1+v_2).
    \end{align}
    With the notation in Lemma \ref{lem-g}, we have
    \begin{align}
      f_{N}(t) - f_{N}(0)
      &= g(N - k + 1;t) - \frac{N-k+1}{L+N-k+1} g(L+N-k+1;t).
    \end{align}
Below  we will prove that the term of $j=0$ in the summation of \eqref{trunc-1-sum}  gives a dominant   contribution and the total contribution of the other terms decays to zero as $M+N\to \infty$, according to  the following  two cases.

    { \bf Case 1}:  $N - k \to \infty$ as $M+N\to \infty$. In this case, statement \ref{lem-g3} in Lemma \ref{lem-g} implies that for sufficiently large $N-k$
    \begin{equation}
      f_{N}(t) - f_{N}(0) = \frac{t^{2}L}{2(L + N - k + 1)} + \bigO((N-k)^{-\frac{1}{4}})
    \end{equation}
    whenever $\abs{t} \leq (N - k)^{\frac{1}{4}}$.  Noting that $|h_{M,N}|\leq  \sqrt{(N-k+1)/M}$ for large $N-k$, we obtain  for $1\leq \abs{t} \leq (N - k)^{\frac{1}{4}}$ that
    \begin{equation}\label{trunc-1-lbound}
      f_{M,N}(t) - f_{M,N}(0)
       \geq \frac{1}{4}t^2-\sqrt{\frac{N-k+1}{M}}|(v_{1}+v_2)t|.
    \end{equation}
    Moreover,  the $t$-function on the right-hand side is increasing over $[1/2,\infty)$ whenever  $\sqrt{(N-k+1)/M} |(v_{1}+v_2)|\leq 1/8$. By introducing
    \begin{equation}
      J_{in}=[1 - k, N-k]\cap [-\lfloor(N - k)^{\frac{1}{4}}\rfloor, \lfloor(N - k)^{\frac{1}{4}} \rfloor],
    \end{equation}
    \begin{equation}
      J_{out}=[1 - k, N-k]\!\setminus\! [-\lfloor (N - k)^{\frac{1}{4}} \rfloor,  \lfloor(N - k)^{\frac{1}{4}} \rfloor],
    \end{equation}
    noting that $v_1, v_2$ are in a compact of $\mathds{R}$, we thus get
      \begin{align}  \label{trunc-1-sum-main}
    \begin{split}
      &
     \Big| \sum_{0\neq j\in J_{in}}e^{-\frac{M}{N-k+1} f_{M,N}(j)} e^{i j ( \phi_{1}-\phi_{2})}\Big| 
     \\
      &\leq 4e^{-\frac{M}{N-k+1} f_{M,N}(0)}  \int_{1/2}^{\infty} dt\,
        e^{-\frac{M}{N-k+1}\big(\frac{1}{4}t^2-\sqrt{\frac{N-k+1}{M}}|(v_{1}+v_2)t|\big) } 
        \\
      &= e^{-\frac{M}{N-k+1} f_{M,N}(0)}  \bigO\big(\sqrt{(N-k+1)/M} \big),
      \end{split}
    \end{align}
     where in the last estimate  use has been made of the change of variables $t \to t\sqrt{(N-k+1)/M}$.


    Noticing the monotonicity  of  $\psi'(x)$ with $x>0$ (cf. \cite[5.15.1]{Olver10}), for all   $t\in [1 - k, N - k]$ we see from
    \begin{align}
      f_{M,N}''(t) &= f_{N}''(t) = (N - k + 1)\big(\psi'(N - k - t + 1) - \psi'(L + N - k - t + 1)\big) \nonumber\\
      &=(N - k + 1)\sum_{j=1}^{L-1}\frac{1}{(N - k - t + 1+j)^2 }> 0 \label{1Lsum}
    \end{align}
  that both $f_{M,N}(t)$ and $f_{N}(t)$ are strictly convex functions.  Thus,
    $f_{M,N}(t)$ attains its minimum over $J_{out}$ at $ t=\lfloor (N - k)^{\frac{1}{4}} \rfloor$ or $t= \lfloor -(N - k)^{\frac{1}{4}} \rfloor$ (when  $k-1\leq \lfloor (N - k)^{\frac{1}{4}} \rfloor$, only the former is possible).
   On  the other hand, it is easy to    see from \eqref{trunc-1-lbound} that
    \begin{equation}
      f_{M,N}(\pm \lfloor (N - k)^{\frac{1}{4}} \rfloor) - f_{M,N}(0)
      \geq \frac{1}{8} \sqrt{N - k+1}
    \end{equation}
    when $M/(N-k+1)$ is sufficiently large. This immediately  implies
    \begin{equation}
      \inf_{t \in  J_{out}} \Re\{f_{M,N}(t) - f_{M,N}(0)\}  \geq  \frac{1}{8} \sqrt{N - k+1},
    \end{equation}
    from which we arrive at
    \begin{equation}\label{trunc-1-sum-out}
      \left| \sum_{ j\in J_{out}}e^{-\frac{M}{N-k+1} f_{M,N}(j)} e^{i j ( \phi_{1}-\phi_{2})}\right|
      \leq N e^{-\frac{M}{N-k+1} f_{M,N}(0) - \frac{M}{ 8\sqrt{N}}}.
    \end{equation}
    Combination of \eqref{trunc-1-sum-main} and \eqref{trunc-1-sum-out} shows that
    \begin{equation}\label{trunc-1-sum1}
      T_{M,N}
      = (z_{1} \bar{z}_{2})^{N - k} e^{-\frac{M}{N-k+1} f_{M,N}(0)}
      \Big(1 + \bigO\Big(\sqrt{\frac{N-k+1}{M}}\Big) + \bigO\big(N e^{- \frac{M}{ 8\sqrt{N} } }\big)\Big).
    \end{equation}

    {\bf Case 2}:   $N - k$ is a fixed number  as $M+N\to \infty$.   In such case $f_{N}(t)$ with $t\in [1-k,N-k]$ attains its unique minimum at $t_0=0$. Moreover,  the convexity of $f_{M,N}(t)$ and $f_{N}(t)$ shows that both  attain the minimum over $\{t \in [1-k,N-k]: |t|\geq 1\}$ at $ t=-1$ or $t= 1$ (when  $N-k=0$, only the former is possible). So we have
    \begin{equation}
      \inf_{|t|\geq 1, t \in [1-k,N-k]} \Re\{f_{N}(t) - f_{N}(0)\}=:\delta_{min}>0,
    \end{equation}
    and  according to the choice of $\rho$ for large $M$
    \begin{equation}
      \inf_{|t|\geq 1, t \in [1-k,N-k]} \Re\{f_{M,N}(t) - f_{M,N}(0)\}\geq \frac{1}{2}\delta_{min}.
    \end{equation}
    \begin{equation}
      \left| \sum_{0\neq j\in[1-k,N-k]}e^{- \frac{M}{N-k+1} f_{M,N}(j)} e^{i j ( \phi_{1}-\phi_{2})}\right|
      \leq e^{- \frac{M}{N-k+1} f_{M,N}(0)} N e^{- \frac{M\delta_{min}}{2(N-k+1)}}.
    \end{equation}
    Therefore, we obtain a similar  estimate as in Case 1
     \begin{equation}\label{trunc-1-sum2}
      T_{M,N} = (z_{1} \bar{z}_{2})^{N - k} e^{-\frac{M}{N-k+1} f_{M,N}(0)}\big (1 + \bigO\big(N e^{- \frac{M\delta_{min}}{2(N-k+1)}})\big).
    \end{equation}

  {\bf Step 2: Asymptotics for the weight function.}

After  a change of variables $s\to s-(N-k+1)$ in   \eqref{truncweight}, the weight function can be rewritten   as
    \begin{align}
      w(z_{1})= \abs{z_{1}}^{-2(N - k + 1)} \int_{- i\infty}^{i\infty} \frac{ds}{2 \pi i} e^{\frac{M}{N-k+1} f_{N}(s) + 2 s \sqrt{\frac{M}{N}} \frac{v_1}{\rho}},
    \end{align}
    where $f_{N}(s)$ is defined in \eqref{trunc-1-fN}.

  By   \eqref{1Lsum} we have  \begin{align}f_{N}''(0)&=(N-k+1) \big(\psi'(N - k + 1) - \psi'(L + N - k + 1)\big)\nonumber\\
  &\geq L/(N-k+1+L)\geq L/(L+N),  \end{align}
 which implies    $f_{N}''(0) M/(N-k+1) \to \infty$  as $M+N\to \infty$ in the setting. Given some small $\delta>0$,
     taking a Taylor expansion at zero, we can  use the standard  steepest descent argument to obtain
     \begin{align}
     \int_{- i\delta}^{i\delta} \frac{ds}{2 \pi i} e^{\frac{M}{N-k+1} (f_{N}(s)-f_{N}(0)) + 2 s \sqrt{\frac{M}{N}} \frac{v_1}{\rho}}
     =\frac{\rho}{2} \sqrt{\frac{N}{M}}
     \Big( \frac{1}{\sqrt{2\pi}}e^{\frac{1}{2} v_{1}^2}
     +  \bigO\Big(\rho \sqrt{\frac{N}{M}}\Big)\Big). \label{trunc-1-weight-2}
\end{align}

On the other hand, we claim the following fact:
 \begin{align}
      \Big |\Big(\int_{i \delta}^{i\infty}+ \int_{-i\infty}^{-i \delta} \Big)\frac{ds}{2 \pi i } e^{\frac{M}{N-k+1} (f_{N}(s)-f_{N}(0)) + 2s \sqrt{\frac{M}{N}} \frac{v_1}{\rho}}  \Big|
      \leq \frac{3(L + N )}{\pi(ML - 1)} \label{trunc-1-weight-1}
    \end{align} for large  $M$.
    If so,    combine  \eqref{trunc-1-weight-2} and  \eqref{trunc-1-weight-1} and  we obtain
    \begin{equation}\label{trunc-1-weight}
      w(z_{1})
      =\frac{\rho \abs{z_{1}}^{-2(N - k + 1)} \sqrt{N}}{2 \sqrt{2\pi M}} e^{\frac{M}{N-k+1} f_{N}(0)-\frac{v_{1}^2}{2}}  \Big(1 +  \bigO\Big(\rho \sqrt{\frac{N}{M}}\Big)\Big).
    \end{equation}

  Furthermore,  we  combine  \eqref{trunc-1-sum1}, \eqref{trunc-1-sum2}  and \eqref{trunc-1-weight} to get an  approximation of the   kernel, uniformly for $v_1, v_2$ in a compact subset of $\mathbb{R}$,        \begin{align}
       K_{M, N}(z_{1}, z_{2}) = \frac{1}{(2 \pi)^{\frac{3}{2}}} e^{- \frac{1}{4} (v_{1}^{2} + v_{2}^{2})}  \frac{(z_{1}\bar{z}_{2})^{N-k}}{\abs{z_{1}z_{2}}^{N-k+1}} \sqrt{\frac{N \rho^{2}}{M}} (1+o(1)),
    \end{align}
from which the theorem   immediately follows.

Finally, let's complete the estimate   \eqref{trunc-1-weight-1}.
    For $y\in \mathds{R}$,  we use the formula \cite[5.8.3]{Olver10} for the ratio of the gamma functions and   obtain
     \begin{align} \begin{split}
      \Re\{ f_N(iy)- f_N(0)\}
      &= -\frac{N-k+1}{2}\sum_{j=0}^{L - 1} \log\Big(1+ \frac{y^2}{(j + N - k + 1)^{2}}\Big)\\
      &\leq  -\frac{N-k+1}{2} \int_{0}^{L} \log\Big(1+ \frac{y^2}{(x + N - k + 1)^{2}}\Big) dx\\
      &= -\frac{(N-k+1) |y|}{2} \int_{\frac{N - k + 1}{|y|}}^{\frac{L + N - k + 1}{|y|}} \log\Big(1+ \frac{1}{x^{2}}\Big) dx. \label{trunc-1-upperbound}\end{split}
    \end{align}
   Next, we analyze  the last integral according to    the range of $|y|$.

    When $\abs{y} < N - k + 1$, using the inequality  $\log(1+x)\geq \frac{1}{2}x$ with $0\leq x<1$ we have   \begin{align}
      \Re\{ f_N(iy)- f_N(0)\}
      &\leq  - \frac{y^{2}}{2} \frac{L}{L + N - k + 1} \leq - \frac{y^{2}}{2} \frac{L}{L + N}.
    \end{align}
    Thus, we obtain
    \begin{align}
      &\abs{\Big(\int_{\delta i}^{(N - k + 1) i} + \int_{-(N - k + 1) i}^{-\delta i}\Big) \frac{ds}{2 \pi i} e^{\frac{M}{N-k+1} (f_{N}(s) - f_{N}(0)) + 2 s \sqrt{\frac{M}{N}} \frac{v_1}{\rho}}}\nonumber\\
      &\leq \frac{1}{\pi} \int_{\delta}^{N - k + 1} e^{- \frac{y^{2}}{2} \frac{M}{N - k + 1} \frac{\tau}{1 + \tau}}\leq \frac{1}{\pi \delta} e^{- \frac{\delta^{2}}{2} \frac{M}{N - k + 1} \frac{L}{L + N}}. \label{trunc-1-weight-11}
    \end{align}

When $N - k + 1 \leq \abs{y} < L + N - k + 1$, divide the integral into two parts and we get
  \begin{align} 
   \int_{\frac{N - k + 1}{|y|}}^{\frac{L + N - k + 1}{|y|}} \log\big(1+ \frac{1}{x^{2}}\big) dx &\geq
    \int_{\frac{N - k + 1}{|y|}}^{1} \log(1+ x^{2}) dx + \int_{1}^{\frac{L + N - k + 1}{|y|}} \log\big(1+ \frac{1}{x^{2}}\big) dx  \nonumber\\
     &\geq
   \frac{1}{2} \int_{\frac{N - k + 1}{|y|}}^{1} x^{2} dx +   \frac{1}{2} \int_{1}^{\frac{L + N - k + 1}{|y|}} \log  \frac{1}{x^{2}} dx
    \nonumber \\  &= \frac{1}{6}  \Big(1- \frac{(N - k + 1)^3}{|y|^3}\Big)+\frac{1}{2}  \Big(1- \frac{|y|}{L + N - k + 1}\Big)\nonumber\\
     &\geq \Big(1- \frac{N - k + 1}{|y|}+1- \frac{|y|}{L + N - k + 1}\Big)\nonumber\\
      &\geq  \frac{1}{6}  \frac{L}{L + N - k + 1}\geq   \frac{1}{6}  \frac{L}{L + N }
    \end{align}
Thus,  we  see from  \eqref{trunc-1-upperbound} that
    \begin{align}
     & \Re\{ f_N(iy)- f_N(0)\}
      \leq  -\frac{L(N-k+1)}{12(L+N)}  |y|,
    \end{align}
    and further have
    \begin{align}
      &\abs{\Big(\int_{(N - k + 1) i}^{(L + N - k + 1) i} + \int_{-(L + N - k + 1) i}^{-(N - k + 1) i}\Big) \frac{ds}{2 \pi i} e^{\frac{M}{N-k+1} (f_{N}(s) - f_{N}(0)) + 2 s \sqrt{\frac{M}{N}} \frac{v_1}{\rho}}}\nonumber\\
      &\leq  24\frac{(L+N)(N-k+1)} {LM}
       e^{-\frac{LM}{12 (L + N)}}.\label{trunc-1-weight-12}
    \end{align}

    When $\abs{y} \geq L + N - k + 1$, noting  that
    \begin{align}
      \Re\{ f_N(iy)- f_N(0)\}
      &< -L (N-k+1)\log \frac{|y|}{L + N - k + 1},
    \end{align}
   we have
    \begin{align}
      &\abs{\Big(\int_{(L + N - k + 1) i}^{i \infty} + \int_{- i \infty}^{-(L + N - k + 1) i}\Big) \frac{ds}{2 \pi i} e^{\frac{M}{N-k+1} (f_{N}(s) - f_{N}(0)) + 2 s \sqrt{\frac{M}{N}} \frac{v_1}{\rho}}} \nonumber\\
      &\leq \frac{L + N - k + 1}{\pi(LM - 1)}\leq \frac{L+N}{\pi (LM-1)}.\label{trunc-1-weight-13}
    \end{align}

    Combination of \eqref{trunc-1-weight-11}, \eqref{trunc-1-weight-12} and \eqref{trunc-1-weight-13} leads to  the estimate   \eqref{trunc-1-weight-1}.
  \end{proof}

  \begin{thm}\label{trunc-2}
    Assume that   $\displaystyle\lim_{M+N \to \infty} M/N=\gamma\in (0,\infty)$ and $\lim_{M+N \to \infty} L/N = \tau\in (0,\infty)$. Given   $q \in [0, 1)$, let $u_{N} = \abs{u_{N}} e^{i \theta}$ with $\theta \in (- \pi, \pi]$ such that $\abs{u_{N}}^{2} = 1 - \frac{1}{N}\lfloor qN \rfloor$, set  \begin{align}
      z_{j} = u_{N}^{M} \Big(1  + \frac{1}{\sqrt{\gamma MN}} \big(v_j-\frac{\gamma \tau}{4 (1 - q)(1 -q+ \tau)}\big)\Big)^{M}, \quad j=1,\ldots,n.
    \end{align}
    Then   limiting  $n$-correlation functions for eigenvalues of $\Pi_M$ in Model B
     \begin{multline}
      \lim_{M+N \to \infty}
        \abs{z_{1}\cdots z_{n}}^2 R^{(n)}_{M,N}(z_{1}, \ldots,z_{n}) \\
       =\begin{cases}
        \det\!\big [  \mathrm{K^{(bulk)}_{\mathrm{crit}}}(\frac{\gamma \tau}{(1 - q)(1 -q+ \tau)};v_i, v_j)\big]_{i,j=1}^n, & q \in (0, 1),\\
       \det\!\big [  \mathrm{K^{(edge)}_{\mathrm{crit}}}(\frac{\gamma \tau}{1 + \tau};v_i, v_j)\big]_{i,j=1}^n, & q = 0,
      \end{cases}
    \end{multline}
     hold   uniformly for $v_1, \ldots,v_n$ in a compact subset of $\mathds{C}$.
   \end{thm}
  \begin{proof}
    {\bf Step 1: Asymptotics for the finite sum.}

Let
 \begin{equation}
        f_{M,N}(t) = M\big(\log\Gamma(N - \lfloor qN \rfloor -t)- \log\Gamma(L + N - \lfloor qN \rfloor -t)\big) + t \log (z_{1}\bar{z}_{2}),   \label{sect-3-f}  \end{equation}
by Cauchy's  residue theorem we rewrite the finite sum  $T_{M,N}$ as
    \begin{align}
      T_{M,N}       &= (z_{1} \bar{z}_{2})^{N- \lfloor qN \rfloor - 1} \oint_{\Sigma}      \frac{dt}{2 \pi i}
       \frac{(-1)^{t} \pi}{ \sin \pi t}e^{-f_{M,N}(t)},
       \label{trunc-2-sum-int}
    \end{align}
    where $\Sigma$ is an anticlockwise contour encircling $ - \lfloor qN \rfloor, 1- \lfloor qN \rfloor \ldots, N-1- \lfloor qN \rfloor$ and is given in  \eqref{crit-sum-contours-bulk}.

    For convenience, let
    \begin{equation}
      \beta=\frac{\gamma \tau}{(1 - q)(1 -q+ \tau)}.
    \end{equation}
    With the choice of $z_1,z_2$, use statement \ref{lem-g3} of Lemma \ref{lem-g} we obtain  for  $  \abs{t} =\bigO( N^{\frac{1}{4}})$
    \begin{align}
      &f_{M,N}(t)-f_{M,N}(0)  =\frac{M}{N - \lfloor qN \rfloor} g(N - \lfloor qN \rfloor, t) - \frac{M}{L + N - \lfloor qN \rfloor} g(L + N - \lfloor qN \rfloor, t) \nonumber\\
      &+t \big(\log(z_{1}\bar{z}_{2}) - M\psi(N - \lfloor qN \rfloor) + M\psi(L + N - \lfloor qN \rfloor)\big)\nonumber\\
      &=\frac{1}{2} (\beta + o(1)) t^{2} +(v_1 + \bar{v}_2 + o(1)) t + \bigO(N^{-\frac{1}{4}}).\label{trucn-2-sum-1}
    \end{align}
    So as $M+N\to \infty$  we have from \eqref{trucn-2-sum-1} that  for $q \in (0, 1)$
    \begin{align}
     \int_{\Sigma_{1}^{+} \cup \Sigma_{1}^{-} }    \frac{dt}{2 \pi i}   \frac{(-1)^{t} \pi}{ \sin \pi t}e^{-f_{M,N}(t)+f_{M,N}(0)}
     & \to \int_{\mathds{R}\pm \frac{i}{8}} \frac{dt}{2 \pi i} \frac{(-1)^{t}\pi}{\sin \pi t} e^{- \frac{1}{2} \beta t^{2} - t (v_1 + \bar{v}_2)}\nonumber\\
      &= \sum_{j = - \infty}^{\infty} e^{- \frac{1}{2} \beta  j^{2} - j(v_1+\bar{v}_2) },
    \end{align} uniformly for $v_1,v_2$ in a compact set of $\mathds{C}$.
 Similarly,  for $q = 0$  we have
    \begin{align}
     &  \int_{\Sigma_{12}^{+} \cup \Sigma_3 \cup \Sigma_{12}^{-} } \frac{dt}{2 \pi i}   \frac{(-1)^{t} \pi}{ \sin \pi t}e^{-f_{M,N}(t)+f_{M,N}(0)} \to \nonumber\\
      & \int_{\Sigma_{3}\cup \{x\pm \frac{i}{8}: x \in [-\frac{1}{2},\infty)\}} \frac{dt}{2 \pi i} \frac{(-1)^{t}\pi}{\sin \pi t} e^{- \frac{1}{2} \beta t^{2} - t (v_1 + \bar{v}_2)}= \sum_{j = 0}^{\infty} e^{- \frac{1}{2} \beta  j^{2} - j(v_1+\bar{v}_2)},
    \end{align}
     uniformly for $v_1,v_2$ in a compact set of $\mathds{C}$.

   The remaining task is to prove that  the remainder of the integral in  \eqref{trunc-2-sum-int} is negligible.  For this, let $t=x+iy$,
    then   for all $ - \lfloor qN \rfloor - \frac{1}{2}\leq x\leq  N - \lfloor qN \rfloor - \frac{1}{2}$ and $|y|\leq \frac{1}{8}$ we get from \cite[5.15.1]{Olver10} that
    \begin{align}
      &\frac{\partial^{2}}{\partial x^{2}} \Re{f_{M,N}(x + iy)}= \Re{f_{M,N}''(x+iy)}\nonumber\\
      &= M\big(\Re{\psi'(N - \lfloor qN \rfloor -(x + iy)) - \Re{\psi'(L + N - \lfloor qN \rfloor -(x + iy))}}\big)\nonumber\\
       &= M \sum_{k=0}^{L-1} \frac{(k + N - \lfloor qN \rfloor -x)^{2} - y^{2}}{((k + N - \lfloor qN \rfloor -x)^{2} + y^{2})^{2}}>0.
    \end{align}
 Combine  asymptotics for the  digamma function $\psi$ and  we obtain for $q\in [0,1)$
    \begin{align}
      &\min_{\lfloor N^{\frac{1}{4}} \rfloor - \frac{1}{2} \leq  x \leq N - \lfloor qN \rfloor - \frac{1}{2}} \frac{\partial}{\partial x} \Re{f_{M,N} (x + iy)}\nonumber\\
      &= - M\Big(\log\abs{1 - \frac{\lfloor N^{1/4} \rfloor - \frac{1}{2} + iy}{N - \lfloor qN \rfloor}} - \log\abs{1 - \frac{\lfloor N^{1/4} \rfloor - \frac{1}{2} + iy}{L + N - \lfloor qN \rfloor}}\Big) + \bigO(1)\nonumber\\
      &> 0
    \end{align}
    and for $q\in (0,1)$
    \begin{align}
      &\max_{- \lfloor qN \rfloor - \frac{1}{2} \leq  x \leq - \lfloor N^{1/4} \rfloor + \frac{1}{2}} \frac{\partial}{\partial x} \Re{f_{M,N} (x + iy)}\nonumber\\
      &=- M\Big(\log\abs{1 + \frac{\lfloor N^{1/4} \rfloor - \frac{1}{2} - iy}{N - \lfloor qN \rfloor}} - \log\abs{1 + \frac{\lfloor N^{1/4} \rfloor - \frac{1}{2} - iy}{L + N - \lfloor qN \rfloor}}\Big) + \bigO(1)\nonumber\\
      &< 0
    \end{align}
    when $M,N$ are   sufficiently large.  By   the above monotonicity and   the asymptotic expansion  \eqref{trucn-2-sum-1},  uniformly for $|y|\leq \frac{1}{8}$ we see that  for $q\in [0,1)$
    \begin{align}
     & \min_{\lfloor N^{\frac{1}{4}} \rfloor - \frac{1}{2} \leq  x \leq  N - \lfloor qN \rfloor - \frac{1}{2}} \Re\{f_{M,N} (x + iy)-f_{M,N}(0)\}\nonumber\\
      &= \Re\{f_{M,N}\big(\lfloor N^{\frac{1}{4}} \rfloor - \frac{1}{2} + iy\big)-f_{M,N}(0)\}= \frac{1}{2}\beta \sqrt{N} (1 + o(1)), \label{trunc-2-est-1}
    \end{align}
    and for $q\in (0,1)$
    \begin{align}
     & \min_{- \lfloor qN \rfloor - \frac{1}{2} \leq x \leq  - \lfloor N^{\frac{1}{4}} \rfloor + \frac{1}{2}} \Re\{f_{M,N} (x + iy)-f_{M,N}(0)\}\nonumber\\
      &= \Re\{f_{M,N}\big(- \lfloor N^{\frac{1}{4}} \rfloor + \frac{1}{2} + iy\big)-f_{M,N}(0)\}
      =  \frac{1}{2}\beta \sqrt{N} (1 + o(1)). \label{trunc-2-est-2}
    \end{align}

On the other hand,   since \eqref{trunc-2-est-1} and  \eqref{trunc-2-est-2} hold true uniformly for $|y|\leq \frac{1}{8}$, we see  when $t\in \Sigma_3\cup \Sigma_5$ that for $q\in [0,1)$
    \begin{align}
      \min_{y \in [-\frac{1}{8}, \frac{1}{8}]} \Re{f_{M,N}(N - \lfloor qN \rfloor - \frac{1}{2} + iy)} = f_{M,N}(0) +\frac{1}{2}\beta  \sqrt{N} (1 + o(1)) \label{trunc-2-est-3}
    \end{align}
    and  for $q\in (0,1)$
     \begin{align}
      \min_{y \in [-\frac{1}{8}, \frac{1}{8}]} \Re{f_{M,N}(- \lfloor qN \rfloor -\frac{1}{2} + iy)} = f_{M,N}(0) +\frac{1}{2}\beta  \sqrt{N} (1 + o(1)). \label{trunc-2-est-4}
    \end{align}

     Note that $|(-1)^t /\sin\pi t|$ has an upper bound independent of $M,N$   whenever $t$ belongs to the chosen contour,   combine \eqref{trunc-2-est-1}-\eqref{trunc-2-est-4} and we know  $q\in [0,1)$
    \begin{align}
     \Big|  \int_{
     \Sigma_{2}^{+}\cup  \Sigma_{3}\cup  \Sigma_{2}^{-}\cup  \Sigma_{5}^{-}\cup  \Sigma_4 \cup \Sigma_{5}^{+}
      } \frac{dt}{2 \pi i}  \frac{(-1)^{t} \pi}{ \sin \pi t}e^{-f_{M,N}(t)+f_{M,N}(0)}\Big| \leq e^{- \frac{1}{4} \beta \sqrt{N}}
    \end{align}
    and  for $q\in (0,1)$
    \begin{align}
     \Big|  \int_{
       \Sigma_{5}^{-}\cup  \Sigma_4 \cup \Sigma_{5}^{+}
      } \frac{dt}{2 \pi i}  \frac{(-1)^{t} \pi}{ \sin \pi t}e^{-f_{M,N}(t)+f_{M,N}(0)}\Big| \leq e^{- \frac{1}{4} \beta \sqrt{N}}
    \end{align}
 when $M, N$ are large.

 In short,  as $M+N\to \infty$  we have  uniformly for $v_1,v_2$ in a compact set of $\mathds{C}$
    \begin{align}\label{trunc-2-asy-sum}
      &e^{-f_{M,N}(0)} (z_{1} \bar{z}_{2})^{\lfloor qN \rfloor-N+1} T_{M,N}\nonumber \\
      &\longrightarrow  \begin{cases}
        \sum_{j = - \infty}^{\infty} e^{- \frac{1}{2} \beta  j^{2} - j(v_1+\bar{v}_2) }, & q \in (0, 1),\\
     \sum_{j = 0}^{\infty} e^{- \frac{1}{2} \beta  j^{2} - j(v_1+\bar{v}_2) }, & q = 0.
      \end{cases}
    \end{align}

 {\bf Step 2: Asymptotics for the weight function.}

In this case,  recalling $f_{M,N}(t)$ in \eqref{sect-3-f} where it is assumed that  $z_1=z_2$, changing  $s\to s- (N - \lfloor qN \rfloor)$
  in the weight function $w(z_{1})$ gives
   \begin{align}
      w(z_{1})= \abs{z_{1}}^{-2(N - \lfloor qN \rfloor)} \int_{- i\infty}^{i\infty} \frac{ds}{2 \pi i}  e^{f_{M,N}(s)}
    \end{align}
    Let $s=iy$, we proceed as in the same way to obtain  \eqref{trunc-1-weight-1}  and  have the following upper bounds
    \begin{align}
      \Big| \Big(\int_{- i \infty}^{-i N^{\frac{1}{4}}} + \int_{i N^{\frac{1}{4}}}^{i \infty}\Big) \frac{ds}{2 \pi i}
      e^{f_{M,N}(s)-f_{M,N}(0)}\Big|=\bigO(\frac{1}{M}). \label{trunc-2-weight-1}
    \end{align}
    When $|s|\leq N^{\frac{1}{4}}$, we can proceed as in \eqref{trucn-2-sum-1}, take a Taylor expansion and get
    \begin{align}
      &   \int_{- i N^{\frac{1}{4}}}^{i N^{\frac{1}{4}}} \frac{ds}{2 \pi i}
      e^{f_{M,N}(s)-f_{M,N}(0)}\rightarrow   
      \frac{1}{\sqrt{2\pi \beta}} e^{-\frac{1}{2\beta}(v_1+\bar{v}_1)^2}.\label{trunc-2-weight-2}
    \end{align}
    Together with \eqref{trunc-2-weight-1}, we have
      \begin{align}
      &e^{-f_{M,N}(0)} \abs{z_{1}}^{2(N - \lfloor qN \rfloor)}  w(z_1)
     \longrightarrow   \frac{1}{\sqrt{2\pi \beta}} e^{-\frac{1}{2\beta}(v_1+\bar{v}_1)^2}.\label{trunc-2-weight-3}
    \end{align}

  Combining \eqref{trunc-2-asy-sum} and \eqref{trunc-2-weight-3}, we get an  approximation of the correlation  kernel, uniformly for $v_1, v_2$ in a compact subset of $\mathbb{C}$,        \begin{align}
       K_{M, N}(z_{1}, z_{2}) = \frac{1}{\sqrt{2\pi^3 \beta}}    \frac{(z_{1}\bar{z}_{2})^{N-\lfloor qN \rfloor-1}}{\abs{z_{1}z_{2}}^{N-\lfloor qN \rfloor}}  (1+o(1))\nonumber\\
       \times  \begin{cases}
         \mathrm{K^{(bulk)}_{\mathrm{crit}}}(\beta;v_1,v_2), & q \in (0, 1),\\
     \mathrm{K^{(edge)}_{\mathrm{crit}}}(\beta;v_1,v_2), & q = 0,
      \end{cases}
    \end{align}
from which the desired result   immediately follows.
  \end{proof}

In order to investigate the scaling limits  in the case that $M/N\to $ as $M+N\to 0$, we need to  introduce a spectral parameter  $u_N$. Given  $q \in (0, 1]$, choose $u_{N} = \abs{u_{N}} e^{i \theta}, \theta \in (-\pi, \pi]$ such that
  \begin{align}
    \Big(1 + \frac{L}{qN}\Big) \abs{u_{N}}^{2} - 1 =0.
  \end{align}
  \begin{thm}\label{trunc-3}
    Assume that $\displaystyle\lim_{M+N \to \infty} M/N=0$ and $\lim_{M+N \to \infty} L/N = \tau\in (0,\infty)$. Let $\rho = \sqrt{\tau}/(1 - \abs{u_N}^{2})$ 
    and introduce
    \begin{align}
      z_{j} = u_{N}^{M} \Big(1  + \frac{v_j}{\rho u_{N}\sqrt{ MN}}\Big)^{M}, \quad j=1,\ldots,n.
    \end{align}
 Then   limiting  $n$-correlation functions for eigenvalues of $\Pi_M$ in Model B
     \begin{multline}
      \lim_{M+N \to \infty}
        \Big(\frac{M}{\rho^{2}|u_N|^2 N}\Big)^{n} \abs{z_{1}\cdots z_{n}}^2   R^{(n)}_{M,N}(z_{1}, \ldots,z_{n}) \\
       =\begin{cases}
        \det\!\big [  \mathrm{K^{(bulk)}_{\mathrm{Gin}}}(v_i, v_j)\big]_{i,j=1}^n, &  q\in (0,1),\\
       \det\!\big [  \mathrm{K^{(edge)}_{\mathrm{Gin}}}(v_i, v_j)\big]_{i,j=1}^n, & q=1,
      \end{cases}
    \end{multline}
     hold   uniformly for $v_1, \ldots,v_n$ in a compact subset of $\mathds{C}$.
  \end{thm}
  \begin{proof}
    For convenience,  let \begin{equation}
      \varphi_j=1+\frac{v_j}{\rho u_{N} \sqrt{MN}}, \quad j=1,2.
    \end{equation}

    {\bf Step 1: Asymptotics for the weight function.}

  Use a  change of variables $s\to -sN$ in   \eqref{truncweight} and the stirling formula, let $s_{0} = q>0$ and the weight function is rewritten   as
    \begin{align}
      w(z_{1})
      = N \Big(1 + \frac{L}{s_{0}N}\Big)^{\frac{M}{2}}  e^{MN f(s_0) - s_0 M N \log \abs{ \varphi_1}^{2} + ML - ML \log N} J,  \label{trunc-3-weight-0}
    \end{align}
    where  for every  $\delta_0>0$
    \begin{align}
      J &=
       \int_{s_0 -i \infty}^{s_0 + i \infty} \frac{ds}{2 \pi i} e^{MN( f(s)-f(s_0)) - (s-s_0) M N \log \abs{ \varphi_j}^{2} - \frac{M}{2} \log \frac{s}{s_0} + \frac{M}{2} \log\frac{L + sN}{L + s_{0}N} + \bigO(\frac{M}{s N})} \nonumber\\
       &=  \Big(\int_{s_0 -i \delta_0}^{s_0 + i \delta_0} \frac{ds}{2 \pi i} +\int_{s_0 - i \infty}^{s_0 -i \delta_0} \frac{ds}{2 \pi i}+ \int_{s_0 +i \delta_0}^{s_0 + i \infty} \frac{ds}{2 \pi i} \Big)\Big(\cdot\Big)=:J_{0}+J_{-}+J_{+},
    \end{align}
    with
    \begin{align}
      f(s) = s \log s - \big(s + \frac{L}{N}\big) \log \big(s + \frac{L}{N}\big) - s \log \abs{u_{N}}^{2}. \label{trunc-3-f}
    \end{align}
    It is easy to see  that  $f'(s)=\log s - \log (s + \frac{L}{N}) - \log\abs{u_{N}}^{2}=0$ has a unique solution $s=s_0$ and  with $s=s_{0}+iy$
     \begin{align}
      &  \frac{\partial}{\partial y} \Re{f (s_0 + iy)}=-\Im f'(s_0 + iy)=-\arctan \frac{y}{s_0} + \arctan \frac{y}{s_0 + \frac{L}{N}}.
    \end{align}
    So $\Re{f (s_0 + iy)}$ with $y\in \mathds{R}$ attains  its  unique  maximum $y=0$, together with the standard steepest decent argument (see e.g. \cite{Wong}) implying that for large $N$ sufficiently
          \begin{align}
 |J_{\pm}|\leq  e^{\frac{1}{2}MN\Re \{f(s_0\pm i\delta_0)-f(s_0)\}}. \label{trunc-3-weight-1}
    \end{align}

    Choose a sufficiently small $\delta_0>0$ and  take a Taylor expansion at $s=s_0$, noting $M/N\to 0$, we have
        \begin{align}
  J_0 &=  \frac{1}{ \sqrt{MN}}\int_{-i \delta_0 \sqrt{MN}}^{ i \delta_0 \sqrt{MN}} \frac{ds}{2 \pi i}
        e^{\frac{1}{2}f''(s_0)s^2 - s  (\frac{v_1}{u}+\frac{\bar{v}_1}{\bar{u}})} \Big(1+ \bigO\Big(\sqrt{\frac{M}{N}}\Big)\Big)\nonumber\\
        &=\frac{1}{ \sqrt{2\pi MN f''(s_{0})} } e^{-\frac{1}{2} (v_1 e^{-i \theta} + \bar{v}_1 e^{i \theta})^{2}} \Big(1+ \bigO\Big(\sqrt{\frac{M}{N}}\Big)\Big). \label{trunc-3-weight-2}
    \end{align}

     Combination of  \eqref{trunc-3-weight-0}, \eqref{trunc-3-weight-1} and \eqref{trunc-3-weight-2}  gives rise to
    \begin{align}
      w(z_{1}) &=
      \abs{ \varphi_1}^{-2MNs_{0}} e^{-\frac{1}{2} (v_1 e^{-i \theta} + \bar{v}_1 e^{i \theta})^{2}} \nonumber \\
       & \times  \Big(1 + \frac{L}{s_{0}N}\Big)^{\frac{M}{2}}\frac{\sqrt{N} e^{MN f(s_{0}) + ML - ML \log N}}{\sqrt{2\pi M f''(s_{0})}}\Big(1+ \bigO\Big(\sqrt{\frac{M}{N}}\Big)\Big), \label{trunc-3-weightest}
    \end{align}
uniformly for $v_1,v_2$ in a compact set of $\mathds{C}$.

 {\bf Step 2: Asymptotics for the finite sum.}

 Write
    \begin{align}
      T_{M,N} = \sum_{j=0}^{N-1} e^{-f_{M,N}(j)},
    \end{align}
    where
 \begin{align}
      f_{M,N}(t) &=M \big(\log \Gamma(L+t+1) -  \log \Gamma(t+1) + t \log \abs{u_{N}}^{2} + t \log (\varphi_1 \bar{\varphi}_2)\big).
    \end{align}
Let $\delta\in (0, q)$, note that for $t \in (0, \delta N)$
    \begin{align} \Re f'_{M,N}(t)
      &= - M\big(\psi(L + t + 1) - \psi(t+1) +
       \log \abs{u_{N}}^{2} + \log |\varphi_1 \bar{\varphi}_2| \big)<0
    \end{align}
when $N$ is large,
we know  that $\Re f_{M,N}(t)$ obtains its minimum at $\delta N$ for $t \in [0, \delta N]$,  so the condition \ref{lem-eulmac-st-1} is satisfied.
%
%
%
%
Apply  the Stirling formula and we obtain
    \begin{align}
      f_{M,N}(tN) &= MNf(t) - t {MN} \log(\varphi_1 \bar{\varphi}_2) - \frac{M}{2} \log \big(1+\frac{L}{tN}\big) + \bigO(\frac{M}{N})\\
      f_{M,N}'(tN) &= M(f'(t) + \bigO((MN)^{-\frac{1}{2}}))
    \end{align}
    uniformly for $t \in [\delta/2,1]$, which implies the condition \ref{lem-eulmac-st0}.

    Recalling   $f(t)$ in \eqref{trunc-3-f}, it is easy to see that $f'(s_0)=0$ and $t=s_0$ is a unique minimum point over $(\delta,1]$.  We thus verify  the condition \ref{lem-eulmac-st1}.

    Now, applying Lemma \ref{lem-eulmac} we get
    asymptotics of the finite sum
    \begin{align}
      T_{M,N}&= (\varphi_1 \bar{\varphi}_2)^{MNs_0} \Big(1 + \frac{L}{s_{0}N}\Big)^{\frac{M}{2}} \frac{\sqrt{N} e^{-MNf(s_{0}) - ML + ML \log N}}{ \sqrt{M f''(s_0)}} \Big(1+ \bigO\Big(\sqrt{\frac{M}{N}}\Big)\Big)\nonumber\\
      &\times  \sqrt{2\pi}  e^{\frac{1}{2} (v_1 e^{-i \theta} + \bar{v}_2 e^{i \theta})^{2}} \begin{cases}1, &  q\in   (0, 1),\\
      \frac{1}{2} \mathrm{erfc}(\frac{v_1 e^{-i \theta} + \bar{v}_2 e^{i \theta}}{\sqrt{2}}), & q=1, \end{cases} \label{trunc-3-totalsum}
    \end{align}
    uniformly for $v_1,v_2$ in a compact set of $\mathds{C}$.


 Putting \eqref{trunc-3-weightest} and  \eqref{trunc-3-totalsum}  together,   we obtain
    \begin{multline}
      K_{M,N}(z_1,z_2)=  \Big(1 + \frac{L}{s_{0}N}\Big)^{M}  \Big( \frac{ z_1 \bar{z}_2 }{|z_1 \bar{z}_2|}\Big)^{MN s_{0}} \frac{N}{M f''(s_{0})} \Big(1+ \bigO\Big(\sqrt{\frac{M}{N}}\Big)\Big)\nonumber\\
      \times  \frac{1}{\pi}  e^{-\frac{1}{2} (|v_1|^2 +|v_2|^2 - v_1 \bar{v}_2) }\begin{cases}1, & q\in   (0, 1),\\
      \frac{1}{2} \mathrm{erfc}(\frac{v_1 e^{-i \theta} + \bar{v}_2 e^{i \theta}}{\sqrt{2}}), &  q=1, \end{cases}
    \end{multline}
     uniformly for $v_1,v_2$ in a compact set of $\mathds{C}$.
     The desired result immediately follows.
 \end{proof}

\section{Products of Ginibre and inverse Ginibre matrices}
\label{sect4}

For two independent  standard complex Ginibre matrices $G_1$ and $G_2$,  Krishnapur  studied  the  quotient  $G_{1}^{-1}G_2$, which is  referred as the spherical ensemble, and  proved   that  the eigenvalues of $G_{1}^{-1}G_2$ form a determinantal point process in the complex plane; see \cite{Kri} and  \cite{HKPV}.    The third product model   under consideration includes the spherical ensemble   as a special   case (see e.g. \cite{Ipsen-Kieburg14}) and  concerns the generalized eigenvalue problem on  two product matrices
$Y_{L} \cdots Y_{1}$ and  $X_{M}\cdots X_{1}$.
More specifically,  it is  defined as
\begin{quote}
{\bf Model C}:   $\Pi_M=(Y_{L} \cdots Y_{1} )^{-1}X_{M}\cdots X_{1}$   where  $X_{1}, \dotsc, X_{M}$, $Y_{1}, \dotsc, Y_{L}$  are  i.i.d.  complex  Ginibre matrices of size $N\times N$, each having  i.i.d.~standard complex Gaussian entries.
\end{quote}
  Its  eigenvalues form a determinantal point process with kernel    \begin{align}
    K_{M, N}(z_{1}, z_{2}) &= \frac{1}{\pi} \sqrt{w(z_{1}) w( z_{2})} T_{M,N},
  \end{align}
  where
  \begin{align}
    w(z)
    = \int_{-c - i \infty}^{-c + i \infty}  \Gamma^{M}(- s) \Gamma^{L}(N + s + 1) \abs{z}^{2s} \label{phereweight}
  \end{align}
  with  $c\in (0, N+1)$ and the finite sum of a truncated series  \begin{align}
    T_{M,N}
    &= \sum_{j=0}^{N-1} \frac{(z_{1}\bar{z}_{2})^{j}}{\Gamma^{M}(j + 1) \Gamma^{L}(N - j)};  \label{pheresum0}
  \end{align}
see \cite{adhikari2013determinantal}.

For Model C, the similar results  as in Model A  are stated  in the following three theorems.  Since their  proofs are parallel to those of  Theorems \ref{ginibre-1},   \ref{ginibre-2} and \ref{ginibre-3},  we may omit some details in the proofs.
  \begin{thm}\label{phere-1}
    Assume that $\displaystyle\lim_{L+M+N \to \infty} (L+M)/N = \infty$. Given  $k \in \{1, \ldots, N\}$, let
    \begin{align}
      \rho 
      &= 2 \big(M\psi'(N - k + 1) + L\psi'(k)\big)^{-\frac{1}{2}}
    \end{align}
    and
    \begin{align}
      z_{j} = e^{\frac{1}{2} (M \psi(N - k + 1) - L \psi(k)) + \frac{v_{j}}{\rho}} e^{i (\theta + \phi_{j})}, \quad j=1, \ldots, n.
    \end{align}
    Then   limiting  $n$-correlation functions for eigenvalues of $\Pi_M$ in Model C
    \begin{align}
      \lim_{L+M+N \to \infty}
        \rho^{-n}  \abs{z_{1}\cdots z_{n}}^2 R^{(n)}_{M,N}(z_{1}, \ldots,z_{n})= \begin{cases}\frac{1}{(\sqrt{2\pi})^3}
        e^{-  \frac{1}{2} v_{1}^{2} }, &n=1,\\
       0, & n>1, \end{cases}
    \end{align}
     hold   uniformly for $v_1, \ldots,v_n$ in a compact subset of $\mathbb{R}$ and $\phi_1, \ldots,\phi_n\in (-\pi,\pi]$.
  \end{thm}
  \begin{proof}

    {\bf Step 1: Asymptotics for the finite sum.}

    Recalling the choice of  $z_{1}$ and  $z_{2}$ in the setting and changing the  summation index $j\to N-k-j$, we rewrite  the finite sum $T_{M,N}$ in \eqref{pheresum0} as
    \begin{align}
      T_{M,N}
      &= (z_{1}\bar{z}_{2})^{N - k} \sum_{j=1-k}^{N-k} e^{- f_{L,M,N}(j)} e^{i j ( \phi_{1}-\phi_{2})}, \label{pheresum}
    \end{align}
    where
    \begin{equation}\label{phere-1-fN}
      \begin{split}
        f_{L,M,N}(t)
        &= M \big(\log\Gamma(N - k - t + 1) + t\psi(N - k + 1)\big) \\
        &\quad + L \big(\log\Gamma(k + t) - t\psi(k)\big) + t \frac{v_{1} + v_{2}}{\rho}.
      \end{split}
    \end{equation}
  We will see    that the term of $j=0$ in  the summation of \eqref{pheresum}  gives a dominant contribution and the total contribution of the other terms is ignorable    as $L+M+N\to \infty$.

    Using  the notation in Lemma \ref{lem-g} we  have
    \begin{align}
      f_{L,M,N}(t) - f_{L,M,N}(0)
      = \frac{M}{N - k+ 1} g(N - k+ 1;t) + \frac{L}{k} g(k; -t) + t\frac{v_{1} + v_{2}}{\rho}.
    \end{align}
    Without loss of generality, we may assume $M/(N - k + 1) \geq L/k$.  Noting  $(L+M)/N \to \infty$, we  derive  that $M/(N - k + 1)\to \infty$ and
    \begin{align}
      \rho^{-1} \leq \frac{C}{2} \sqrt{M\psi'(N-k+1)}
    \end{align}
    for some positive constant $C$ independent of $L,M,N$.  Also noting  that $g(k, -t)\geq 0$  for all $t \in[1-k,N-k]$ we have
    \begin{equation}\label{phere-1-ineq1}
      \begin{split}
        &f_{L,M,N}(t) - f_{L,M,N}(0)\\
        &\geq \frac{M}{N - k+ 1} g(N - k+ 1; t) - \frac{C}{2} \sqrt{M\psi'(N-k+1)} |(v_{1}+v_2)t|,
      \end{split}
    \end{equation}
 from which    we can obtain a  similar lower bound to \eqref{gin1-lbound} as in the proof of Theorem \ref{ginibre-1}. Noting
  \begin{align}
      f''_{L,M,N}(0) =
      M\psi'(N - k + 1) + L\psi'(k),
    \end{align}
 we further use    the similar arguments  between \eqref{gin1-lbound}--\eqref{ginibre-1-sum-out}  in the proof of Theorem \ref{ginibre-1} to get
    \begin{align}
      T_{M,N} = (z_{1} \bar{z}_{2})^{N - k} e^{-f_{L,M,N}(0)}\Big (1 +
        \bigO\big(\sqrt{\frac{N-k+1}{M}} \big)\Big). \label{phere-1-sum1}
    \end{align}

    {\bf Step 2: Asymptotics for the weight function.}

    After  a change of variables $s\to s-(N-k+1)$ in   \eqref{phereweight}, the weight function can be rewritten   as
    \begin{align}
      w(z_{1})= \abs{z_{1}}^{- 2(N - k + 1)} \int_{-i\infty}^{\infty} \frac{ds}{2 \pi i} e^{\tilde{f}_{M,L,N}(s) + 2s\frac{v_{1}}{\rho}},
    \end{align}
    where
    \begin{align}
      \tilde{f}_{L,M,N}(s)
      &= M (\log\Gamma(N - k - s + 1) + s\psi(N - k + 1)) \nonumber\\
      &\quad + L (\log\Gamma(k + s) - s\psi(k)).
    \end{align}
    Let $s=iy$,  we see from statement \ref{lem-g2} in Lemma \ref{lem-g} that
    \begin{equation}
      \begin{split}
        &\Re{\tilde{f}_{L,M,N}(iy)} - \Re{\tilde{f}_{L,M,N}(0)}\\
        &= \frac{M}{N - k + 1} \Re{g(N - k + 1; iy)} + \frac{L}{k} \Re{g(k ;-iy)}\\
        &\leq-M \frac{\abs{y}}{4} \min\Big\{1, \frac{\abs{y}}{N - k + 1}\Big\} - L \frac{\abs{y}}{4} \min\Big\{1, \frac{\abs{y}}{k}\Big\}.
      \end{split}
    \end{equation}
   Without loss of generality, we may consider the case that  $M/(N - k + 1) \geq L/k$. For any $\delta > 0$, it follows that
    \begin{align}
      & \Big |\Big(\int_{i \delta}^{i\infty}+ \int_{-i\infty}^{-i \delta} \Big)\frac{ds}{2 \pi i } e^{(\tilde{f}_{L,M,N}(s)-\tilde{f}_{L,M,N}(0)) + 2s \frac{v_1}{\rho}}  \Big|
      \nonumber\\
      &\leq  2 \Big(\int_{ \delta}^{N -  k+1}dy\,  e^{-\frac{M}{4(N - k+1)} y^{2}}  + \int_{N -  k+1}^{\infty}dy\,  e^{-\frac{M}{8}  y}\Big)\nonumber\\
      &\leq
        \frac{2(N-k+1)}{M\delta } e^{-\frac{M\delta^2}{4(N - k+1)}  }  +    \frac{8}{M} e^{-\frac{1}{8}M(N-k+1)}, \label{phere-1-weight-1}
    \end{align}
    which decays  exponentially to zero.
Meanwhile, take a Taylor expansion of $\tilde{f}_{L,M,N}(s)$ at $s=0$  origin, rescale  the  variable and we obtain
    \begin{align}
      &\int_{-i\delta}^{i\delta} \frac{ds}{2 \pi i} e^{(\tilde{f}_{M,L,N}(s) - \tilde{f}_{M,L,N}(0)) + 2s\frac{v_{1}}{\rho}}\nonumber \\
    &  = \frac{1}{\sqrt{\tilde{f}_{M,L,N}''(0)}} \frac{1}{\sqrt{2 \pi }}
     e^{- \frac{v_{1}^2}{2}}  \Big(1 +  \bigO\Big(\frac{1}{\sqrt{ M \psi'(N - k + 1)}}\Big)\Big).\label{phere-1-weight-2}
    \end{align}
    Combination of  \eqref{phere-1-weight-1} and \eqref{phere-1-weight-2} gives rise to
    \begin{equation}\label{phere-1-weight}
      w(z_{1})
      =\frac{\abs{z_{1}}^{-2(N - k + 1)} }{ \sqrt{2\pi \tilde{f}_{M,L,N}''(0)}} e^{\tilde{f}_{L,M,N}(0)-\frac{v_{1}^2}{2}}  (1 +  o(1)).
    \end{equation}

    Combine  \eqref{phere-1-sum1} and \eqref{phere-1-weight}, we get an  approximation of the correlation  kernel, uniformly for $v_1, v_2$ in a compact subset of $\mathbb{R}$,
    \begin{align}
       K_{M, N}(z_{1}, z_{2}) = \frac{1}{(2 \pi)^{\frac{3}{2}}} e^{- \frac{1}{4} (v_{1}^{2} + v_{2}^{2})}  \frac{(z_{1}\bar{z}_{2})^{N-k}}{\abs{z_{1}z_{2}}^{N-k+1}}   \frac{1+o(1)}{\sqrt{\tilde{f}_{M,L,N}''(0)} },
    \end{align}
    from which   Theorem \ref{phere-1}   immediately follows.
  \end{proof}

  \begin{thm}\label{phere-2}
    Assume that $\displaystyle\lim_{L+M+N \to \infty} M/N=\gamma_{1}\in [0,\infty)$ and $\displaystyle\lim_{L+M+N \to \infty} L/N=\gamma_{2}\in [0,\infty)$ with $\gamma_{1} + \gamma_{2} > 0$.
  Given   $q \in (0, 1)$, let  $u_{N} = \abs{u_{N}} e^{i \theta}$ with $\theta \in (- \pi, \pi]$ such that \begin{equation} \abs{u_{N}}^{2} = \Big(1 - \frac{1}{N}\lfloor qN \rfloor\Big)^{\frac{M}{M+L}}  \Big( \frac{1}{N}+  \frac{1}{N}\lfloor qN \rfloor\Big)^{-\frac{L}{M+L}}, \end{equation}
  and  set  \begin{align}
      z_{j} = N^{M-L} u_{N}^{M+L}    \Big(1  + \frac{v_j-
      \frac{\gamma_{1}}{4(1-q)} - \frac{\gamma_{2}}{4q}   }{\sqrt{(\gamma_{1}+\gamma_{2}) (M+L)N}}\Big)^{M+L}, \quad j=1,\ldots,n.
    \end{align}
 Then   limiting  $n$-correlation functions for eigenvalues of $\Pi_M$ in Model C
     \begin{multline}
      \lim_{M+N \to \infty}
        \abs{z_{1}\cdots z_{n}}^2 R^{(n)}_{M,N}(z_{1}, \ldots,z_{n}) \\
       =\det\!\big [  \mathrm{K^{(bulk)}_{\mathrm{crit}}}(\frac{\gamma_{1}}{1-q} + \frac{\gamma_{2}}{q};v_i, v_j)\big]_{i,j=1}^n
    \end{multline}
     hold   uniformly for $v_1, \ldots,v_n$ in a compact subset of $\mathds{C}$.
   \end{thm}
   \begin{proof}  {\bf Step 1: Asymptotics for the finite sum.}

     By Cauchy's  residue theorem,  we rewrite the finite sum  $T_{M,N}$ as
    \begin{align}
      T_{M,N}
      &= (z_{1} \bar{z}_{2})^{N- \lfloor qN \rfloor - 1} \oint_{\Sigma}      \frac{dt}{2 \pi i} \frac{(z_{1} \bar{z}_{2})^{-t} (-1)^{t} \pi}{\Gamma^{M}(N - \lfloor qN \rfloor - t) \Gamma^{L}(\lfloor qN \rfloor + t + 1) \sin \pi t}, \label{phere-2-sum-int}
    \end{align}
    where $\Sigma$ is an anticlockwise contour  just encircling    $ - \lfloor qN \rfloor, 1- \lfloor qN \rfloor \ldots, N-1- \lfloor qN \rfloor$. Here we choose it as   in  \eqref{crit-sum-contours-bulk}.

    Using \eqref{asy-psi}  and part \ref{lem-g3} in Lemma \ref{lem-g}, as $N\to \infty$ we have for   $\abs{t} \leq N^{\frac{1}{4}}$
    \begin{align}
       & \log\bigg( \frac{\Gamma^{M}(N - \lfloor qN \rfloor) \Gamma^{L}(\lfloor qN \rfloor + 1) (z_{1} \bar{z}_{2})^{-t}}{\Gamma^{M}(N - \lfloor qN \rfloor - t) \Gamma^{L}(\lfloor qN \rfloor + t + 1)}\bigg)\nonumber\\
        &=-\frac{1}{2}\Big(\frac{M}{N - \lfloor qN \rfloor} + \frac{L}{\lfloor qN \rfloor + 1}\Big)t^{2} - (v_{1} + \bar{v}_{2}+o(1)) t + \bigO(N^{-\frac{1}{4}}). \label{phere-asy-crit-1}
    \end{align}
   So  with the chosen  contour in mind, as $L+M+N\to \infty$  we have from \eqref{phere-asy-crit-1} that
    \begin{align}
      & \int_{\Sigma_{1}^{+} \cup   \Sigma_{1}^{-} } \frac{dt}{2 \pi i} \frac{\Gamma^{M}(N - \lfloor qN \rfloor) \Gamma^{L}(\lfloor qN \rfloor + 1) (z_{1} \bar{z}_{2})^{-t} (-1)^{t}\pi}{\Gamma^{M}(N - \lfloor qN \rfloor - t)  \Gamma^{L}(\lfloor qN \rfloor + t + 1) \sin \pi t}\nonumber\\
      &\to \int_{\mathds{R}\pm \frac{i}{8}} \frac{dt}{2 \pi i} \frac{(-1)^{t}\pi}{\sin \pi t} e^{- \frac{1}{2} \beta t^{2} - t (v_1+
       \bar{v}_2)}=   \sum_{j = - \infty}^{\infty} e^{- \frac{1}{2} \beta  j^{2} - j(v_1+\bar{v}_2) },
    \end{align}
    uniformly for $v_1,v_2$ in a compact set of $\mathbb{C}$, where
    \begin{equation}
      \beta:=\frac{\gamma_{1}}{1-q} + \frac{\gamma_{2}}{q}.
    \end{equation}

    The remaining task is to prove that  the remainder of the integral in  \eqref{phere-2-sum-int} is negligible. Taking part \ref{lem-g3} in Lemma \ref{lem-g} into consideration, we see that
    \begin{align}
      &\max_{t \in \Sigma \backslash \Sigma_{1}^{\pm}} \abs{\frac{\Gamma^{M}(N - \lfloor qN \rfloor) \Gamma^{L}(\lfloor qN \rfloor + 1) (z_{1} \bar{z}_{2})^{-t}}{\Gamma^{M}(N - \lfloor qN \rfloor - t) \Gamma^{L}(\lfloor qN \rfloor + t + 1)}}
      \leq e^{-\frac{1}{2} \beta \sqrt{N}}
    \end{align}
    holds for sufficiently large $N$. Note that $|(-1)^t /\sin\pi t|$ has an upper bound independent of $L,M,N$   whenever $t$ belongs to the chosen contour,  we thus obtain
    \begin{align}
      \abs{\int_{\Sigma \backslash \Sigma_{1}^{\pm}} \frac{dt}{2 \pi i} \frac{\Gamma^{M}(N - \lfloor qN \rfloor) \Gamma^{L}(\lfloor qN \rfloor + 1) (z_{1} \bar{z}_{2})^{-t} (-1)^{t}}{\Gamma^{M}(N - \lfloor qN \rfloor - t) \Gamma^{L}(\lfloor qN \rfloor + t + 1) \sin \pi t}} \leq e^{- \frac{1}{4} \beta \sqrt{N}}
    \end{align}
    when $N$ is large.

 In short,  as $N\to \infty$  we have  uniformly for $v_1,v_2$ in a compact set of $\mathbb{C}$
    \begin{align}\label{phere-2-sum}
      &\Gamma^{M}(N - \lfloor qN \rfloor) \Gamma^{L}(\lfloor qN \rfloor + 1) (z_{1} \bar{z}_{2})^{\lfloor qN \rfloor-N+1} T_{M,N}\nonumber \\
      &\longrightarrow \sum_{j = - \infty}^{\infty} e^{- \frac{1}{2} \beta  j^{2} - j(v_1+\bar{v}_2) }.
    \end{align}

 {\bf Step 2: Asymptotics for the weight function.}

 Changing  $s\to -s- (N - \lfloor qN \rfloor)$
  in the weight function $w(z_{1})$ gives
   \begin{align}
      w(z_{1})= \abs{z_{1}}^{-2(N - \lfloor qN \rfloor)} \int_{- i\infty}^{i\infty} \frac{ds}{2 \pi i}
      \Gamma^{M}(N - \lfloor qN \rfloor - s) \Gamma^{L}(\lfloor qN \rfloor + s + 1)  |z_1|^{2s},
    \end{align}
    Let $s=iy$, using statement \ref{lem-g2} in Lemma \ref{lem-g}, we can derive
    \begin{align}
      &\abs{\frac{\Gamma^{M}(N - \lfloor qN \rfloor - s) \Gamma^{L}(\lfloor qN \rfloor + s + 1)}{\Gamma^{M}(N - \lfloor qN \rfloor) \Gamma^{L}(\lfloor qN \rfloor + 1)} \abs{z_{1}}^{2s}} \nonumber\\
      &\leq \exp\Big\{- \frac{M\abs{y}}{4} \min\Big\{1, \frac{\abs{y}}{N - \lfloor qN \rfloor}\Big\} - \frac{L\abs{y}}{4} \min\Big\{1, \frac{\abs{y}}{\lfloor qN \rfloor+1}\Big\}\Big\},
    \end{align}
    from which we know  that
    \begin{align}
      &\abs{\Big(\int_{- i \infty}^{-i N^{\frac{1}{4}}} + \int_{i N^{\frac{1}{4}}}^{i \infty}\Big) \frac{ds}{2 \pi i} \frac{\Gamma^{M}(N - \lfloor qN \rfloor - s) \Gamma^{L}(\lfloor qN \rfloor + s + 1)}{\Gamma^{M}(N - \lfloor qN \rfloor) \Gamma^{L}(\lfloor qN \rfloor + 1)} \abs{z_{1}}^{2 s}}\nonumber\\
      &\leq \int_{N^{\frac{1}{4}}}^{N - \lfloor qN \rfloor}dy  e^{-\frac{M}{4(N - \lfloor qN \rfloor)} y^{2}}  + \int_{N - \lfloor qN \rfloor}^{\infty}dy  e^{-\frac{M}{8}  y} \nonumber\\
      &\leq N^{-\frac{1}{4}}e^{-\frac{M\sqrt{N}}{4(N - \lfloor qN \rfloor)}  }  +    \frac{8}{M} e^{-\frac{1}{8}M(N - \lfloor qN \rfloor)}. \label{phere-2-weight-1}
    \end{align}
    When $|s|\leq N^{\frac{1}{4}}$, we see from  \eqref{phere-asy-crit-1} that
    \begin{align}
      &\int_{- i N^{\frac{1}{4}}}^{i N^{\frac{1}{4}}} \frac{ds}{2 \pi i} \frac{\Gamma^{M}(N - \lfloor qN \rfloor - s) \Gamma^{L}(\lfloor qN \rfloor + s + 1)}{\Gamma^{M}(N - \lfloor qN \rfloor) \Gamma^{L}(\lfloor qN \rfloor + 1)} \abs{z_{1}}^{2 s} \nonumber\\
      &\longrightarrow   \int_{- i \infty}^{i \infty} \frac{ds}{2 \pi i} e^{\frac{1}{2} \beta s^{2} + s(v_1+\bar{v}_1)}=\frac{1}{\sqrt{2\pi \beta}} e^{-\frac{1}{2\beta}(v_1+\bar{v}_1)^2}.
    \end{align}
    Together with \eqref{phere-2-weight-1}, we have
      \begin{align}
      &\frac{\abs{z_{1}}^{2(N - \lfloor qN \rfloor)}}{\Gamma^{M}(N - \lfloor qN \rfloor) \Gamma^{L}(\lfloor qN \rfloor + 1)}  w(z_1)
     \longrightarrow   \frac{1}{\sqrt{2\pi \beta}} e^{-\frac{1}{2\beta}(v_1+\bar{v}_1)^2}.\label{phere-2-weight-2}
    \end{align}

    Combining \eqref{phere-2-sum} and \eqref{phere-2-weight-2}, we get an  approximation of the correlation  kernel, uniformly for $v_1, v_2$ in a compact subset of $\mathbb{C}$,
    \begin{align}
       K_{M, N}(z_{1}, z_{2}) = \frac{(z_{1}\bar{z}_{2})^{N-\lfloor qN \rfloor-1}}{\abs{z_{1}z_{2}}^{N-\lfloor qN \rfloor}} \mathrm{K^{(bulk)}_{\mathrm{crit}}}(\beta;v_1,v_2) (1+o(1)),
    \end{align}
    from which   Theorem \ref{phere-2}   immediately follows.
   \end{proof}

   \begin{thm}\label{phere-3}
     Assume that $\displaystyle\lim_{L+M+N \to \infty} (M+L)/N = 0$.  Given  $q \in (0, 1)$,  choose an $M$- and $N$-dependent  spectral variable  $u = \abs{u} e^{i \theta}$ with $\theta \in (- \pi, \pi]$  such that
      \begin{align}
     \abs{u}^{2} = q^{\frac{M}{L+M}} (1-q)^{-\frac{L}{L+M}}.
   \end{align}
     Let
     \begin{align}
       \rho
       = \frac{q(1-q)(L+M)}{\abs{u} ((1-q)M + qL)}
     \end{align}
     and
     \begin{align}
       z_{j} = N^{\frac{M-L}{2}} u^{L+M} \Big(1 + \frac{v_{j}}{u \rho\sqrt{(L+M)N}}\Big)^{L+M}.
     \end{align}
 Then   limiting  $n$-correlation functions for eigenvalues of $\Pi_M$ in Model C
     \begin{multline}
      \lim_{L+M+N \to \infty}
        \Big(\frac{M}{\rho \abs{u} N}\Big)^{n} \abs{z_{1}\cdots z_{n}}^2   R^{(n)}_{M,N}(z_{1}, \ldots,z_{n}) \\
       =\det\!\big [  \mathrm{K^{(bulk)}_{\mathrm{Gin}}}(v_i, v_j)\big]_{i,j=1}^n
    \end{multline}
     hold   uniformly for $v_1, \ldots,v_n$ in a compact subset of $\mathbb{C}$.
   \end{thm}
   \begin{proof}
     For convenience,  let
     \begin{align}
       \varphi_{j} = 1 + \frac{v_{j}}{u \rho\sqrt{(L+M)N}}, \quad j=1,2.
     \end{align}

     {\bf Step 1: Asymptotics for the weight function.}
     Let  $s_{0} = q$ and
      \begin{align}
      f(s) = \frac{M}{L+M}s(\log s - 1) + \frac{L}{L+M} (1-s)(\log(1-s) - 1) - s\log\abs{u}^{2}. \label{phere-3-f}
    \end{align}
     By a  change of variables $s\to -sN$ in   \eqref{phereweight} and the Stirling formula,      simple calculation leads to
     \begin{align}
        w(z_{1})
        &= N e^{C_{M,N} - s_{0} (L+M)N \log\abs{\varphi_{1}}^{2} } J, \label{phere-3-weight-0}
    \end{align}
    where \begin{equation}C_{M,N}=(L+M)N f(s_{0})    +  \frac{M}{2} \log\frac{2\pi}{s_0N} + \frac{L}{2} \log(2\pi N(1-s_{0}) ) \end{equation}
    and   for some $\delta_{0} > 0$
    \begin{align}
      J &= \int_{s_{0}-i\infty}^{s_{0}+i\infty} \frac{ds}{2 \pi i} e^{(L+M)N (f(s) - f(s_{0}) - (s-s_{0}) \log\abs{\varphi_{1}}^{2}) - \frac{M}{2} \log\frac{s}{s_{0}}  + \frac{L}{2} \log\frac{1-s}{1-s_{0}}+ \bigO(\frac{L+M}{N})} \nonumber\\
       &=  \Big(\int_{s_0 -i \delta_0}^{s_0 + i \delta_0} \frac{ds}{2 \pi i} +\int_{s_0 - i \infty}^{s_0 -i \delta_0} \frac{ds}{2 \pi i}+ \int_{s_0 +i \delta_0}^{s_0 + i \infty} \frac{ds}{2 \pi i} \Big)\Big(\cdot\Big)=:J_{0}+J_{-}+J_{+}.
    \end{align}

    It is easy to see  that
    \begin{align}
      f'(s) = \frac{M}{L+M} \log s - \frac{L}{L+M} \log(1-s) - \log\abs{u}^{2} = 0
    \end{align}
    has a unique solution $s=s_0$ and with $s=s_{0}+iy$
    \begin{align}
      \frac{\partial}{\partial y} \Re{f (s_0 + iy)}
      &=-\Im f'(s_0 + iy)\\
      &=-\frac{M}{L+M} \arctan \frac{y}{s_0} - \frac{L}{L+M} \arctan\frac{y}{1 - s_{0}}.
    \end{align}
    So $\Re{f (s_0 + iy)}$ with $y\in \mathds{R}$ attains  its  unique  maximum $y=0$, together with the standard steepest decent argument (see e.g. \cite{Wong}) implying that for large $N$ sufficiently
    \begin{align}
      |J_{\pm}|\leq  e^{\frac{1}{2}MN\Re \{f(s_0\pm i\delta_0)-f(s_0)\}}. \label{phere-3-weight-1}
    \end{align}
    Choose a sufficiently small $\delta_0>0$ and  take a Taylor expansion at $s=s_0$, noting $(L+M)/N\to 0$, we have
    \begin{align}
      J_0 &=  \frac{1}{ \sqrt{(L+M)N}}\int_{-i \delta_0 \sqrt{(L+M)N}}^{ i \delta_0 \sqrt{(L+M)N}} \frac{ds}{2 \pi i}
      e^{\frac{1}{2}f''(s_0)s^2 - s  (\frac{v_1}{u\rho}+\frac{\bar{v}_1}{\bar{u}\rho})} \Big(1+ \bigO\Big(\sqrt{\frac{L+M}{N}}\Big)\Big)\nonumber\\
      &=\frac{1}{ \sqrt{2\pi (L+M)N f''(s_{0})} } e^{-\frac{1}{2} (v_1 e^{-i \theta} + \bar{v}_1 e^{i \theta})^{2}} \Big(1+ \bigO\Big(\sqrt{\frac{L+M}{N}}\Big)\Big). \label{phere-3-weight-2}
    \end{align}
    Combination of  \eqref{phere-3-weight-0}, \eqref{phere-3-weight-1} and \eqref{phere-3-weight-2}  gives rise to
    \begin{align}\label{phere-3-weightest}
        w(z_{1}) &=e^{C_{M,N}- s_{0} (L+M)N \log\abs{\varphi_{1}}^{2} }\nonumber\\
        &\times \frac{\sqrt{N}}{ \sqrt{2\pi (L+M) f''(s_{0})} } e^{-\frac{1}{2} (v_1 e^{-i \theta} + \bar{v}_1 e^{i \theta})^{2}} \Big(1+ \bigO\Big(\sqrt{\frac{L+M}{N}}\Big)\Big),
    \end{align}
    uniformly for $v_1,v_2$ in a compact set of $\mathbb{C}$.

    {\bf Step 2: Asymptotics for the finite sum.}
     Let
    \begin{align}
      f_{M,N}(t)
      &=M\big(\log\Gamma(t+1) - t\log(N\abs{u}^{2}) - t \log (\varphi_1 \bar{\varphi}_2)\big)\nonumber\\
      &\quad + L\big(\log\Gamma(N-t) + t \log N - t\log\abs{u}^{2} - t\log (\varphi_1 \bar{\varphi}_2)\big),
    \end{align}
the finite sum is then rewritten as
    \begin{align}
      T_{M,N}
      = \sum_{j=0}^{N-1} e^{-f_{M,N}(j)}.
    \end{align}
       For any $\delta\in (0,q)$,  when  $N$ is sufficiently  large  we have  for $t \in (0, \delta N)$
    \begin{align}
      \Re{f_{M,N}'(t)}
      &= M \big(\psi(t + 1) - \log(N\abs{u}^{2}) - \log \abs{\varphi_1 \bar{\varphi}_2}\big)\nonumber\\
      &\quad + L\big(-\psi(N-t) + \log (N \abs{u}^{2}) - \log \abs{\varphi_1 \bar{\varphi}_2}\big)\nonumber\\
      &<0.
    \end{align}
    This shows that $\Re f_{M,N}(t)$ obtains its minimum at $\delta N$ for $t \in [0, \delta N]$,  so the condition \ref{lem-eulmac-st-1} is satisfied.
Meanwhile, apply   the Stirling formula and we obtain
    \begin{multline}
      f_{M,N}(tN) = (L+M)Nf(t)
       - (L+M)(t-s_{0}) \log ( \varphi_1 \bar{\varphi}_2)
       + \frac{M}{2} \log\frac{t}{s_{0}}  \\- \frac{L}{2} \log\frac{1-t}{1-s_{0}}+C_{M,N}-L\log(N(1-s_0))+M\log(Ns_0)+ \bigO(\frac{L+M}{N})\\
    \end{multline}
    and   \begin{equation}   f_{M,N}'(tN) = (L+M)\big( f'(t) + \bigO(((L+M)N)^{-\frac{1}{2}})\big)  \end{equation}
    uniformly for $t \in [\delta/2,1]$ and $v_1, v_2$ in a compact set of $\mathbb{C}$, which implies the condition \ref{lem-eulmac-st0}.  Recalling  the definition of $f(t)$  in \eqref{phere-3-f}, it is easy to see  that $f'(s_0)=0$ and $t=s_0$ is a unique minimum point over $(\delta,1]$. Besides, \begin{equation}f''(s_{0}) =\frac{M}{L+M}\frac{1}{s_0}+
    \frac{L}{L+M}\frac{1}{1-s_0}
    > 0.\end{equation} Thus the condition  \ref{lem-eulmac-st1} is  also satisfied.

    Now, apply Lemma \ref{lem-eulmac} and
    we obtain asymptotics of the finite sum
    \begin{align}
      T_{M,N}&= e^{-C_{M,N}+L\log(N(1-s_0))-M\log(Ns_0) + s_{0}(L+M)N\log (\varphi_1 \bar{\varphi}_2)}\nonumber\\
      &\times \frac{\sqrt{2\pi N}}{ \sqrt{(L+M) f''(s_{0})}}  e^{\frac{1}{2} (v_1 e^{-i \theta} + \bar{v}_2 e^{i \theta})^{2}} \Big(1+ \bigO\Big(\sqrt{\frac{L+M}{N}}\Big)\Big), \label{phere-3-totalsum}
    \end{align}
    uniformly for $v_1,v_2$ in a compact set of $\mathbb{C}$.

    Putting \eqref{phere-3-weightest} and  \eqref{phere-3-totalsum}  together,   we obtain
    \begin{align}
      K_{M,N}(z_1,z_2)&=  \frac{(N(1-s_{0}))^{L}}{(N s_{0})^{M}}  \Big( \frac{ z_1 \bar{z}_2 }{|z_1 \bar{z}_2|}\Big)^{(L+M)Ns_{0}} \frac{N}{Mf''(s_{0})} \Big(1+ \bigO\Big(\sqrt{\frac{M}{N}}\Big)\Big)\nonumber\\
      &\times  \frac{1}{\pi}  e^{-\frac{1}{2} (|v_1|^2 +|v_2|^2 - v_1 \bar{v}_2) }
    \end{align}
    uniformly for $v_1,v_2$ in a compact set of $\mathds{C}$. Note that $\rho |u|f''(s_0)=1$, the desired result immediately follows.
 \end{proof}

\begin{rem}
  Considering  the mixed products of  truncated unitary matrices and their inverses (cf. \cite{adhikari2013determinantal}), we can prove the  same results  as in Model C for   the fourth product model
  \begin{quote}
{\bf Model D}:  $\Pi_M=(Y_{L} \cdots Y_{1} )^{-1} X_{M}\cdots X_{1}$   where  $Y_1, \ldots, Y_L, X_{1}, \dotsc, X_{M}$ are  i.i.d.  truncated unitary matrices, chosen  from  the  upper left $N\times N$ corner  of a random   unitary  $(N + L) \times (N + L)$ matrix.  \end{quote}
\end{rem}

\section{Concluding remarks} \label{conclusion}

\subsection{Crossovers}

Since  the parameter  $\beta$ via $\gamma$ represents  the   limit of the ratio  $M/N$,   we definitely have reason to believe that  as $\beta\to 0$  the  critical bulk and  edge kernels,  defined in   \eqref{critbulk} and  \eqref{critiedge},    tend to the Ginibre bulk and edge  kernels, while  as  $\beta \to  \infty$
we observe  the Gaussian phenomenon. Indeed, we have
\begin{prop} \label{crossprop}
The following hold  uniformly for $z_1, z_2$ in a compact subset of $\mathds{C}$.
\begin{itemize}
\item[(i)]
  \begin{equation}
    \lim_{\beta \to 0}
    \beta\,
     \mathrm{K^{(bulk)}_{crit}}\big(\beta;  \sqrt{\beta} z_1,  \sqrt{\beta}  z_2\big)
      =  \mathrm{K^{(bulk)}_{Gin}}\big( z_1,  z_2\big) \label{cross-1}
     \end{equation}
and
  \begin{equation}
    \lim_{\beta \to \infty}  \sqrt{\beta}\,
     \mathrm{K^{(bulk)}_{crit}}\big(\beta;  \sqrt{\beta} z_1,  \sqrt{\beta}  z_2\big)
      = \frac{1}{\sqrt{2\pi^3}}   e^{- \frac{1}{2} (|z_1|^2+|\bar{z}_2|^2  + z_1^{2} + \bar{z}_2^2)} ;  \label{cross-2}
     \end{equation}
  \item[(ii)]
   \begin{equation}
    \lim_{\beta \to 0}
    \beta\,
     \mathrm{K^{(edge)}_{\mathrm{crit}}}\big(\beta;  \sqrt{\beta} z_1,  \sqrt{\beta}  z_2\big)
      =  \mathrm{K^{(edge)}_{Gin}}\big( z_1,  z_2\big) \label{cross-3}
     \end{equation}
and
  \begin{equation}
    \lim_{\beta \to \infty}  \sqrt{\beta}\,
     \mathrm{K^{(edge)}_{crit}}\big(\beta;  \sqrt{\beta} z_1,  \sqrt{\beta}  z_2\big)
      = \frac{1}{\sqrt{2\pi^3}}   e^{- \frac{1}{2} (|z_1|^2+|\bar{z}_2|^2  + z_1^{2} + \bar{z}_2^2)} .\label{cross-4}
     \end{equation}
  \end{itemize}
  \end{prop}

\begin{proof}[Proof of Proposition \ref{crossprop}]   Change  variables and rewrite the kernel as
\begin{multline}
    \sqrt{\beta} \mathrm{K^{(bulk)}_{\mathrm{crit}}}(\beta;\sqrt{\beta} z_1,  \sqrt{\beta}  z_2)=\\ \frac{1}{\sqrt{2 \pi^{3}}}
     e^{- \frac{1}{4} (|z_1|^2+|\bar{z}_2|^2  + z_1^{2} + \bar{z}_2^2)}
    \sum_{j= - \infty}^{\infty} e^{- \frac{1}{2} \beta j^{2} -  \sqrt{\beta}(z_1+ \bar{z}_2)j}.  \label{rescalingcritbulk}   \end{multline}

    We first deal with   part (i).
When  $\beta$  is large sufficiently, we have    $|(z_1+ \bar{z}_2)/\sqrt{\beta}|\leq 1/4$ uniformly  for $z_1, z_2$ in a compact subset of $\mathds{C}$.  Thus,
\begin{align}
  \Big |  \sum_{0\neq j= - \infty}^{\infty} e^{- \frac{1}{2} \beta j^{2} - \sqrt{\beta}  (z_1+ \bar{z}_2)j}\Big |\leq 2\sum_{j= 1}^{\infty} e^{- \frac{1}{2} \beta j (j -\frac{1}{2})}\leq \frac{2}{e^{\frac{\beta}{4}}-1},     \end{align}
  from which  and \eqref{rescalingcritbulk}  as $\beta\to \infty$ we  obtain \eqref{cross-2}.    To prove \eqref{cross-1},
  use the duality formula to rewrite
  \begin{multline}
    \beta \mathrm{K^{(bulk)}_{\mathrm{crit}}}(\beta;\sqrt{\beta} z_1,  \sqrt{\beta}  z_2)=\\ \frac{1}{ \pi}
     e^{- \frac{1}{2} (|z_1|^2+|\bar{z}_2|^2  -2z_1 \bar{z}_2)}
    \sum_{j= - \infty}^{\infty} e^{- \frac{2\pi^2}{ \beta} j^{2} +  \frac{2\pi i}{\sqrt{\beta}}(z_1+ \bar{z}_2)j}.    \end{multline}
As $\beta\to 0$, use the same argument as above we  can prove that the summation on the  right-hand side tends to 1 and we thus obtain  \eqref{cross-1}.

    For part (ii), we can use the same argument as to \eqref{cross-2} to  prove \eqref{cross-4}. To prove \eqref{cross-3},
    rewrite the kernel as
\begin{multline}
    \beta \mathrm{K^{(edge)}_{\mathrm{crit}}}(\beta;\sqrt{\beta} z_1,  \sqrt{\beta}  z_2)=\\ \frac{1}{\sqrt{2 \pi^{3}}}
     e^{- \frac{1}{4} (|z_1|^2+|\bar{z}_2|^2  + z_1^{2} + \bar{z}_2^2)}
    \sum_{j= - \infty}^{\infty}     \sqrt{\beta} e^{- \frac{1}{2} (\sqrt{\beta} j)^{2} -  \sqrt{\beta}j (z_1+ \bar{z}_2)},   \end{multline}
    and treat the summation as an integral over $(0,\infty)$,  as $\beta\to 0$  we have
    \begin{equation}
    \sum_{j= - \infty}^{\infty}     \sqrt{\beta} e^{- \frac{1}{2} (\sqrt{\beta} j)^{2} -  \sqrt{\beta}j (z_1+ \bar{z}_2)} \to
   \int_{0}^{\infty}e^{-\frac{1}{2}t^2-(z_1+ \bar{z}_2)t}dt.  \end{equation}
  Thus   the proof of  \eqref{cross-3} is complete.
\end{proof}

\subsection{Random block  non-Hermitian matrices}   Following \cite{Gud, burda2010spectrum,AB12},
the matrix product  $ \Pi_{M}^{}=X_{M}\cdots X_{2}X_{1}$  admits a  linearization (see also \cite{Mingo-Speicher17})  \begin{equation}
Y=\left(
\begin{array}{cccccc}
 0& X_{M-1}& 0& 0 &\cdots  & 0\\
0& 0&  X_{M-2}& 0 &\cdots  & 0\\
0& 0&  0& X_{M-3} &\cdots& 0\\
\vdots & \vdots & \vdots  & \vdots & \ddots & \vdots \\
0& 0&  0 & 0 &\cdots& X_{1}\\
X_{M} & 0&  0 & 0 &\cdots  &0
\end{array}
\right ).
\end{equation}
Here  the  cyclic block matrix $Y$  is constructed from $X_1, \ldots, X_M$  which   are placed in a cyclic positions
of  a sparse $MN\times MN$ matrix.
Note that  $\det(zI_{MN} -Y)=\det(z^M I_{N} -X_M \cdots X_1)$,  we see that all $MN$  eigenvalues of $Y$ can be read off from $N$ eigenvalues  of  the product $ \Pi_{M}^{}$  by taking the $M$-th roots for each one. Thus,  we guess that  eigenvalues of $Y$ have similar local statistical properties as stated in    Section \ref{sect1.4}.   We will come back to this question in the future.

As  non-Hermitian  analogies  of
 random (Hermitian) band matrices, random non-Hermitian band matrices are relatively less studied.  Although there is  a lot of literature   available about  random band matrices,  the local  bulk statistics  transition   from Poisson to  GOE/GUE, which is conjectured that a phase transition occurs as the bandwidth  $W \propto \sqrt{N}$, remains a major open problem; see   \cite{Shch} and references therein.   As to the  non-Hermitian case,  say  $Y$, the key results in  Section \ref{sect1.4}  provide compelling evidence since  the bandwidth  $M \propto \sqrt{MN}$ is equivalent to   $M \propto N$. So we conjecture  that  the critical phenomena  in Theorem \ref{ginibre-2} characterize the crossover in  random non-Hermitian  complex band  matrices.

\subsection{Open questions}  As mentioned in the  Introduction, there are     a few     primary  problems on eigenvalues of products of non-Hermitian random matrices, in which  local universality problem seems particularly outstanding.
We believe that our method can be generalized to deal with  products of rectangular complex Ginibre matrices and products of truncated unitary matrices from unitary matrices of different  size,  under certain assumptions on the growth of matrix sizes; see \cite{LWW} for a similar treatment on singular values.     But the other problems may be  challenging.
Here we list some   to conclude this last section.

\begin{que}
  Consider the product of i.i.d. real/quaternion Gaussian random   matrices (or truncated orthogonal/symplectic matrices) and prove a similar phase transition under  the same rescalings; cf. \cite{forrester16}.
\end{que}

Recently, a class of random matrix ensembles  is   introduced in \cite{FKK}  and is named after P\'{o}lya  ensembles because of their relations  with P\'{o}lya frequency functions.   These ensembles unify almost all  exact  complex matrix ensembles and have lots of similar structures with complex Ginibre ensembles.   The second question is
\begin{que}   Verify Theorems \ref{ginibre-1}--\ref{ginibre-3} or find new local statistics for products of i.i.d. P\'olya ensembles; cf. \cite{KK19,FKK}.
\end{que}

\begin{que}
  Verify Theorems \ref{ginibre-1}--\ref{ginibre-3} for  complex eigenvalues products  of non-Hermitian random matrices with i.i.d. complex entries     under certain higher moment assumptions.
\end{que}

\section*{Acknowledgement}



We acknowledge support by the National Natural Science Foundation of China \#11771417, the Youth Innovation Promotion Association CAS \#2017491 (D.-Z. Liu), and by   the National Natural Science Foundation of China \#11901161 (Y. Wang).


\end{document}